\documentclass[preprint,11pt]{elsarticle}

\usepackage{amssymb}
\usepackage{amsthm}

\usepackage{amsmath}
\newcommand{\nn}{\nonumber}
\newcommand{\p}{\partial}
\newcommand{\bx}{\mathbf{x}}
\newcommand{\bn}{\mathbf{n}}
\newcommand{\bu}{\mathbf{u}}
\newcommand{\Div}{\nabla\!\cdot\!}
\usepackage{tikz}
\usetikzlibrary{arrows.meta}
\usetikzlibrary{decorations.markings}
\tikzset{->-/.style={decoration={
  markings,
  mark=at position #1 with {\arrow{>}}},postaction={decorate}}}
  \tikzset{middlearrow/.style={
        decoration={markings,
            mark= at position 0.55 with {\arrow{#1}} ,
        },
        postaction={decorate}
    }
}          
\newtheorem{lemma}{Lemma}[section]
\newtheorem{proposition}{Proposition}[section]
\newtheorem{remark}{Remark}[section]
\usepackage{geometry}
\numberwithin{figure}{section}
\numberwithin{equation}{section}
\allowdisplaybreaks

\journal{Journal de Math\'ematiques Pures et Appliqu\'ees}

\begin{document}

\begin{frontmatter}



\title{Extended water wave systems of Boussinesq equations on a finite interval: Theory and numerical analysis}

\author{Dionyssios Mantzavinos\,$^a$ and Dimitrios Mitsotakis\,$^{b,}$\corref{cor}}

\address[ad1]{Department of Mathematics, University of Kansas, U.S.A.}

\address[ad2]{Department of Mathematics and Statistics, Victoria University of Wellington, New Zealand
\\[10mm]
{\normalsize Dedicated to the memory of Vassilios A. Dougalis}
}

\begin{abstract}
Considered here is a class of Boussinesq systems of Nwogu type. Such systems describe propagation of nonlinear and dispersive water waves of significant interest such as solitary and tsunami waves. 
The initial-boundary value problem on a finite interval for this family of systems is studied both theoretically and numerically. First, the linearization of a certain generalized Nwogu system is solved analytically via the unified transform of Fokas. The corresponding analysis reveals two types of admissible boundary conditions, thereby suggesting  appropriate boundary conditions for the nonlinear Nwogu system on a finite interval. 
Then, well-posedness is established, both in the weak and in the classical sense, for a regularized Nwogu system in the context of an initial-boundary value problem that describes the dynamics of water waves in a basin with wall-boundary conditions. 
In addition, a new modified Galerkin method is suggested for the numerical discretization of this regularized system in time, and its convergence   is proved along with optimal error estimates. Finally, numerical experiments illustrating the effect of the boundary conditions on the reflection of solitary waves by a vertical wall are also provided. 
\end{abstract}


\cortext[cor]{dimitrios.mitsotakis@vuw.ac.nz}

\end{frontmatter}



\section{Introduction}

Nonlinear and dispersive water waves are described by the Euler system of equations~\cite{Whitham2011}. As Euler's equations still remain one of the hardest problems to solve (even numerically), various simplified systems of partial differential equations have been suggested as alternative approximations, especially for long waves of small amplitude because of their applications. Waves of this type are also called weakly nonlinear and weakly dispersive waves. Tsunamis, solitary waves, internal waves and even atmospheric waves fall into the regime of weakly nonlinear and weakly dispersive waves \cite{Whitham2011}. The aforementioned simplified systems should obey the laws of physics and mathematics; importantly, they should be well-posed when supplemented with physically sound boundary conditions (Newton's principle of determinacy), admit classical solitary wave solutions, preserve symmetries and reasonable forms of invariants such as energy, and agree with  laboratory experiments. Compliance with such fundamental laws is along the lines of scientific rigor and justifies the use of such systems in practical applications.

After the pioneering work of Boussinesq \cite{Bous1871,Bous1872}, several Boussinesq systems have been introduced to describe the propagation of weakly nonlinear and weakly dispersive waves, including Peregrine's system~\cite{Per67}. Peregrine's system, also known as classical Boussinesq system in the case of flat bottom topography, was rederived in one dimension along with a whole class of asymptotically equivalent systems known as the $abcd$-systems \cite{BCS2002}. The Cauchy problem on the real line for this class of Boussinesq systems was studied in~\cite{Schonbek1981,BCS2004,saut2012cauchy}. In scaled and non-dimensional variables, the $abcd$ family of systems takes the form
\begin{equation}\label{eq:coeffs}
\begin{aligned}
\eta_t+u_x+\varepsilon (\eta u)_x+\sigma^2 \left(au_{xxx}-b\eta_{xxt}\right)=0\ , \\
u_t+\eta_x+\varepsilon uu_x+\sigma^2 \left(c\eta_{xxx}-du_{xxt}\right)= 0\ ,
\end{aligned}
\end{equation}
where $x$ denotes the horizontal spatial independent variable, $t$ is the time, $\eta=\eta(x,t)$ is the free surface elevation above a flat bottom located at depth $D_0=-1$, $u=u(x,t)$ is the horizontal velocity of the fluid measured at depth $\theta D_0$ with $\theta\in [0,1]$, $\varepsilon, \sigma$ are small parameters characterizing the nonlinearity and the dispersion of the waves, and $a,b,c,d$ are parameters given by
\begin{equation*}
\begin{aligned}
&a=\tfrac{1}{2}\left(\theta^2-\tfrac{1}{3}\right)\nu, 
\quad
b=\tfrac{1}{2}\left(\theta^2-\tfrac{1}{3}\right)(1-\nu),
\\
&c=\tfrac{1}{2}\left(1-\theta^2\right)\mu,
\quad
d=\tfrac{1}{2}\left(1-\theta^2\right)(1-\mu),
\end{aligned}
\quad
\mu,\nu\in \mathbb{R},
\end{equation*}
so that $a+b+c+d=1/3$.
The parameters $\varepsilon$ and $\sigma^2$ in the Boussinesq regime are of the same order, and the Stokes number is $S=\varepsilon/\sigma^2=O(1)$. Hence, as usual, throughout this work we take $\varepsilon=\sigma^2$ for simplicity and without loss of generality. 

Some important representatives of the $abcd$ family of systems \eqref{eq:coeffs} include systems that are well-posed and also admit classical solitary waves as special solutions. Examples of such systems are the  Bona-Smith system ($a=0$, $b>0$, $c\leq0$, $d>0$), which includes the coupled Benjamin-Bona-Mahony  BBM-BBM system as a special case ($c=0$), the classical Boussinesq system ($a=b=c=0$, $d>0$), the \textit{Nwogu system} ($a<0$, $b=c=0$, $d>0$), and the \textit{regularized Nwogu system}  or ``reverse'' Bona-Smith system  ($a<0$, $b>0$, $c=0$, $d>0$). It should be noted that the Nwogu system  was derived in \cite{Nwogu93} as an alternative to the Peregrine system with the ability to choose the coefficients $a,b,c,d$ so that the linear dispersion relation becomes optimal compared to the corresponding linear dispersion relation of the Euler equations.

In practical applications and numerical simulations, systems like those mentioned above are posed in bounded domains. This fact highlights the need for appropriately formulated initial-boundary value problems. An important issue in this direction is the choice of appropriate boundary conditions. In particular, in the case of both the regularized and the non-regularized Nwogu systems, it is \textit{not a priori clear}  how many boundary values --- and of which type --- must be specified as data for a well-posed problem on the finite interval.  
One of the main results of this work is the identification of appropriate boundary conditions on the finite interval for the following \textit{generalized Nwogu system}:
\begin{equation}\label{lrn-i}
\begin{aligned}
\eta_t+u_x+\varepsilon (\eta u)_x+\varepsilon \left(a u_{xxx}-b \eta_{xxt}\right) = 0\ ,\\
u_t+\eta_x+\varepsilon uu_x -\varepsilon du_{xxt}=0\ ,
\end{aligned}
\quad 
(x, t) \in \left(-L, L\right) \times \left(0, T\right)\ ,
\end{equation}
with $a<0$, $b\geq 0$, $c=0$, $d>0$,  which contains both the regularized and the original Nwogu system, for $b>0$ and $b=0$ respectively.

Appropriate boundary conditions for the generalized system \eqref{lrn-i} are identified through solving the \textit{linear} counterpart of that system via the \textit{unified transform}. This method was first introduced by Fokas in \cite{f1997} (see also the monograph \cite{fokas2008} and the review article \cite{dtv2014}) and has since been employed for the analysis of linear as well as initial-boundary value problems in various settings --- see, for example,~\cite{fi2004,FP2005,ff2008,ffss2009,fl2012,fp2015,hm2015a,hm2015b,fhm2017,hm2020}. Recent developments have led to the advancement of the unified transform to \textit{systems} of PDEs, in particular via the work \cite{dgsv2018}. 
Exploiting the framework laid out in \cite{dgsv2018} along with recent progress noted in  \cite{jgm2021} on the linearization of the classical Boussinesq equation on the half-line, here we employ the unified transform to derive a \textit{novel, explicit solution formula} for the linearization  of the generalized Nwogu system \eqref{lrn-i} on a finite interval:
\begin{equation}\label{lng-ibvp-i}
\begin{aligned}
\eta_t + u_x + \varepsilon \left(a  u_{xxx} - b \eta_{xxt}\right) = 0\ , 
\\
u_t + \eta_x - \varepsilon d u_{xxt} = 0\ ,
\end{aligned}
\quad (x, t) \in (-L,L) \times (0, T)\ .
\end{equation}
While the initial conditions accompanying this system are the usual ones, namely $\eta(x, 0)$ and $u(x, 0)$ given, we perform our analysis without initially specifying any boundary conditions. Instead, we discover the conditions that lead to an explicit solution formula (and hence to a well-formulated problem) through the application of the unified transform. 
In particular, our analysis indicates that one of the following two pairs of boundary values must supplement system \eqref{lng-ibvp-i} as boundary conditions  (see also Remark~\ref{adm-bc-r}):
\begin{equation}\label{adm-bc-i}
\left\{u(\pm L, t), u_{xx}(\pm L, t)\right\} \ \ \text{or} \ \ \left\{\eta(\pm L, t), u_x(\pm L, t)\right\}.
\end{equation}
This finding provides strong theoretical evidence that the boundary conditions \eqref{adm-bc-i} should lead to a well-posed problem for the nonlinear generalized Nwogu system \eqref{lrn-i} and, in particular, for the original Nwogu system on a finite interval.

We note that our analysis, via the unified transform, of the linear system \eqref{lng-ibvp-i} is carried out for \textit{nonzero} boundary conditions of the form \eqref{adm-bc-i}. Nevertheless, one of the most important initial-boundary value problems for Nwogu-type systems is the one with wall-boundary conditions on the boundaries of a basin.  The well-posedness of this initial-boundary value problem with reflection boundary conditions for the Peregrine system and its linearization was studied in \cite{FP2005, Adamy2011, jgm2021}. Specifically, it was proven that for Peregrine's system only the classical homogeneous Dirichlet \textit{wall-boundary conditions}
\begin{equation}\label{eq:wallbc}
    u(-L,t)=u(L,t)=0\ ,
\end{equation} 
are required for well-posedness, ensuring that there is no flow through the boundaries. Analogous initial-boundary value problems have been studied in detail for Bona-Smith systems in \cite{ADM2009}, while their special case of BBM-BBM systems was studied in \cite{BC1998}. There, it was shown that in addition to the wall-boundary condition~\eqref{eq:wallbc} homogeneous Neumann boundary conditions for  $\eta$ are also required:
\begin{equation}\label{eq:wallbc2}
    \eta_x(-L,t)=\eta_x(L,t)=0\ .
\end{equation}
Although these conditions are not satisfied by Peregrine's system, they are satisfied by the solutions of the Euler equations when wall-boundary conditions are imposed \cite{Khakimzyanov2018a} (see also Appendix \ref{sec:appendix}). Thus, satisfying both boundary conditions \eqref{eq:wallbc} and \eqref{eq:wallbc2}  reflects  a more accurate description of water waves in a basin.

The physical relevance of homogeneous (zero) boundary conditions as illustrated above motivates the study of well-posedness for the nonlinear \textit{regularized Nwogu system}, namely system  \eqref{lrn-i} with $b>0$, supplemented with such conditions on a finite interval. In particular, the second main result of this work establishes the well-posedness of that system in the case of reflective boundary conditions (the Dirichlet  problem was analyzed in \cite{ADM2009}). 
We note that, although in this work we restrict ourselves to the case of a flat bottom, it is anticipated that the variable bottom case can be handled by using similar concepts.
The particular initial-boundary value problem with reflective boundary conditions that we consider for the regularized Nwogu system can be written in dimensionless and scaled variables as
\begin{equation}\label{eq:Nwogu3nd}
\begin{aligned}
&\begin{aligned}
&\eta_t+u_x+\varepsilon (\eta u)_x+\varepsilon \left(a u_{xxx}-b \eta_{xxt}\right) =0\ ,\\
&u_t+\eta_x+\varepsilon uu_x -\varepsilon du_{xxt}=0\ , 
\end{aligned} \quad (x, t) \in (-L,L) \times (0, T)\ ,\\
& \eta(x,0)=\eta_0(x),\quad u(x,0)=u_0(x)\ ,\\
& u(-L,t)=u(L,t)=0\ ,\quad  u_{xx}(-L,t)=u_{xx}(L,t)=0\ ,
\end{aligned}
\end{equation}
where $a<0$ and $b,d>0$ such that $a+b+d=1/3$. As noted earlier, in the case where $a=0$ and $b,d>0$, the system reduces to the BBM-BBM system which was analyzed extensively in \cite{BC1998}, while for $a<0$, $b=0$ and $d>0$ the system becomes the well-known Nwogu system \cite{Nwogu93}. Observe that, because of the boundary conditions on $u$, the second (momentum) equation in \eqref{eq:Nwogu3nd} yields the additional conditions \eqref{eq:wallbc2}, which are satisfied implicitly, and thus there is no need for them to be explicitly stated. Furthermore, as shown in  Appendix \ref{sec:appendix}, the second set of boundary conditions
$u_{xx}(\pm L,t)=0$ in \eqref{eq:Nwogu3nd}
are also satisfied by the solutions of Euler's equations.

After  proving that the nonlinear system \eqref{eq:Nwogu3nd} is well-posed in the Hadamard sense locally in time, we study its numerical discretization with Galerkin/finite element method. Wall-boundary conditions for the numerical solution of the Nwogu  system were first suggested in \cite{walkley1999,WB1999,WB2002} but \textit{without} theoretical justification. In the special case of the linearized Nwogu system, one can establish well-posedness with the particular wall-boundary conditions using Galerkin approximations \cite{CWW2004}. The presence of the third-order spatial derivative $u_{xxx}$ in Nwogu-type systems, like the term $\eta_{xxx}$ in the case of the Bona-Smith system, makes their numerical discretization with Galerkin methods challenging. This difficulty, for example, can be related to the requirement of well-defined second derivative of the velocity component $u$ of the numerical solution. While Lagrange elements  guarantee only local smoothness, smooth cubic splines (at least) are left to be used for the standard Galerkin method accompanied with suboptimal convergence results \cite{ADM2010, DMS2007}. As a remedy to this problem, a modification of the standard Galerkin method for the Nwogu system was suggested in \cite{WB1999}, allowing the use of Lagrange elements. While the convergence of that particular method is still unclear, a similar modified Galerkin method was studied and proven to be convergent in the case of the Bona-Smith system \cite{DMS2010}. Here, we develop the analogous modified Galerkin/finite element method for the regularized Nwogu system \eqref{eq:Nwogu3nd}, and we show that its semidiscrete Galerkin approximations converge to the analytical solutions of \eqref{eq:Nwogu3nd} with optimal convergence rate.

\vskip 2mm
\noindent
\textbf{Structure.} This work is organized as follows. 
In Section \ref{sec:linear}, using the unified transform of Fokas  we obtain an explicit solution formula for the linearization \eqref{lng-ibvp-i}  of the generalized Nwogu system~\eqref{lrn-i} on the finite interval $(-L, L)$. Our analysis shows that this problem (and hence its nonlinear counterpart) is well-formulated if either $\left\{u(\pm L,t), u_{xx}(\pm L,t)\right\}$ or $\left\{\eta(\pm L,t), u_x(\pm L,t)\right\}$ are prescribed as boundary conditions.
In Section \ref{sec:wp-s}, we establish local Hadamard well-posedness for the regularized Nwogu system~\eqref{eq:Nwogu3nd}   in the case of  wall-boundary conditions $u(\pm L,t)=u_{xx}(\pm L,t)=0$ (the well-posedness of \eqref{eq:Nwogu3nd} with  $\eta(\pm L,t)=u_x(\pm L,t)=0$ was proved in \cite{ADM2009}). 
In Section \ref{num-s}, we provide the derivation and analysis of a modified Galerkin method for the numerical solution of the regularized system~\eqref{eq:Nwogu3nd}. The convergence of the method is also verified experimentally, while a demonstration of the reflection of solitary waves illustrates the practical use of the initial-boundary value problem with wall-boundary conditions. 
Finally, some brief concluding remarks are given in Section~\ref{conc-s}, while the use of the particular set of wall-boundary conditions is justified in Appendix \ref{sec:appendix} by showing that  solutions to the Euler equations satisfy the same wall-boundary conditions even in the case of variable bottom topography.

\vskip 2mm
\noindent
\textbf{Notation.}
Throughout this work, we denote by $L^2:=L^2(-L,L)$  the Hilbert space of measurable, square-integrable real-valued functions on $(-L,L)$.
For any integer $s\geq 0$, we denote by $H^s:=H^s(-L,L)$   the classical Sobolev space of $s$-times weakly differentiable functions on $(-L, L$), 
$$H^s=\left\{v\in L^2: \partial_x^j v\in L^2 \text{ for all } j=0,1,\ldots, s\right\}\ ,
$$
accompanied with the usual norm 
$\left\|v\right\|_s := \left(\sum_{j=0}^s \int_{-L}^L \left|\partial_x^j v(x)\right|^2 dx \right)^{1/2}$, 
where $\partial_x^j$ denotes the $j$-th partial derivative with respect to $x$.
Note that $H^0=L^2$, while the norm of $L^2$ will be denoted by $\|\cdot\|$. For $m\geq 0$, we will also consider the Banach space $C^s:=C^s(-L,L)$ of real-valued $s$-times continuously differentiable functions defined on $[-L,L]$, equipped with the norm
$$\left\|v\right\|_{C^s}:=\sup_{0\leq j\leq s}~ \sup_{x\in [-L,L]}\left|\partial_x^j v(x)\right|\ .
$$
We write $A \lesssim B$ if $A \leq C B$ with $C>0$ a  constant independent of discretization parameters such as~$\Delta x$ or other crucial parameters.

\section{Explicit solution of the linear problem}\label{sec:linear}

In this section, we employ the unified transform of Fokas in order to solve the linear counterpart~\eqref{lng-ibvp-i} of the generalized Nwogu system \eqref{lrn-i} on the finite interval $(-L,L)$. Our analysis reveals that the combinations  \eqref{adm-bc-i}
%
%
are two possible choices of \textit{admissible boundary conditions} for this linear problem. As such, these two combinations should also lead to a well-posed problem  at the nonlinear level, supporting the theoretical and numerical findings of Sections \ref{sec:wp-s} and \ref{num-s}. The first set of data in \eqref{adm-bc-i}  describes wall-boundary conditions and so it can be used for studying the reflection of water waves on a vertical wall, while the second set of data corresponds to the wave maker problem.

Importantly, we note that: (i) The calculations of this section remain  valid for  $b = 0$, which is the value corresponding to the original  Nwogu system \eqref{eq:Nwogu3nd}. In particular, setting $b = 0$ (equivalently, $\beta=0$) in the solution formulas \eqref{sols-comb} yields the corresponding solutions to the linearization of the Nwogu system (see problem \eqref{lnwog-ibvp} in Remark \ref{non-reg-r}). (ii) Furthermore, our analysis and the resulting solution formulas  hold for general \textit{nonzero} boundary conditions.
(iii) In addition, since swapping $\eta$ with $u$ transforms the generalized Nwogu system  into the Bona-Smith system, the  formulas \eqref{sols-comb} derived here provide the solution also for the linearized Bona-Smith system formulated with nonzero boundary conditions on a finite interval.
%
\subsection{Derivation of the global relation}
We begin by noting that throughout this section we assume sufficient smoothness and decay as needed for our computations to hold.
Setting $\alpha = -a \varepsilon >0$, $\beta=b \varepsilon \geqslant 0$, $\delta = d \varepsilon>0$ allows us to write the linear generalized Nwogu  system~\eqref{lng-ibvp-i} in the form
\begin{equation}\label{lng-ibvp}
\begin{aligned}
\eta_t + u_x - \alpha u_{xxx} - \beta \eta_{xxt} = 0\ , 
\\
u_t + \eta_x - \delta u_{xxt} = 0\ ,
\end{aligned}
\quad (x, t) \in (-L,L) \times (0, T)\ .
\end{equation}
While we supplement  system \eqref{lng-ibvp} with the usual initial conditions
\begin{equation}\label{lng-ic}
\eta(x,0)=\eta_0(x),\quad u(x,0)=u_0(x)\ ,
\end{equation}
we do not yet specify any boundary conditions as it is not a priori clear what choices of boundary data  lead to a well-formulated problem. Instead, we introduce the notation 
\begin{equation}\label{lng-bv}
\begin{aligned}
&g_j(t) := \p_x^j u(-L, t), \ h_j(t) := \p_x^j u(L, t), \  j=0, 1, 2\ ,
\\
&g_3(t) := \eta(-L, t), \ h_3(t) := \eta(L, t)\ ,
\end{aligned}
\end{equation}
for the various boundary values that arise in our analysis and defer the prescription of some of these as boundary conditions to a later point.

Let the finite-interval Fourier transform pair of a function $f\in L^2(-L,L)$ be defined by
\begin{equation}\label{ft-def}
\widehat  f(k) = \int_{x=-L}^L e^{-ikx} f(x) dx, \quad k\in \mathbb C,  
\quad
f(x) = \frac{1}{2\pi} \int_{k\in \mathbb R} e^{ikx} \widehat  f(k) dk, \quad  x\in (-L,L)\ ,
\end{equation}
and note that, since $x$ is bounded,  $\widehat f(k)$ is an entire function of $k$ via a Paley-Wiener-type theorem (e.g. see Theorem 7.2.3 in \cite{s1994}).
Taking the Fourier transform \eqref{ft-def} of  system \eqref{lng-ibvp} while noting that the second component of \eqref{lng-ibvp} yields
$\eta_x (-L, t) = \delta g_2'(t) - g_0'(t)$
and
$ \eta_x (L, t) = \delta h_2'(t) - h_0'(t)$,
we obtain the vector ODE
\begin{equation}\label{lng-ft}
 \widehat {\mathbf{v}}_t(k, t) + M(k) \, \widehat {\mathbf{v}}(k, t) 
=
\mathbf{A}(k, t), 
\quad
k \in \mathbb C \setminus \left\{\pm \tfrac{i}{\sqrt \delta}, \pm \tfrac{i}{\sqrt \beta}\right\}, 
\end{equation}
where $\mathbf{v} = \left(\eta, u\right)^T$, $M$ is a $2\times 2$ matrix given by
\begin{equation*}
M(k)
=
ik\left(\begin{array}{lr}
0 & \left(1+\alpha k^2\right)\left(1+\beta k^2\right)^{-1} 
\\
\left(1+\delta k^2\right)^{-1} & 0
\end{array}\right)\ ,
\end{equation*}
and $\mathbf{A} = \left(A, B\right)^T$ is a vector with components
\begin{align}\label{A-vec-def}
A(k, t) 
&= 
\frac{1}{1+\beta k^2}
\Big\{
-\left(1+\alpha k^2\right) \left[e^{-ikL} h_0(t) - e^{ikL} g_0(t)\right]
\nn\\
&\quad
+  i \alpha k \left[ e^{-ikL} h_1(t)- e^{ikL} g_1(t)\right]
+ \alpha \left[e^{-ikL} h_2(t) - e^{ikL} g_2(t)\right]
\Big\}
\nn\\
&\quad
+
\frac{\beta}{1+\beta k^2} 
\Big\{
\delta \left[e^{-ikL} h_2''(t) - e^{ikL} g_2''(t) \right]
-\left[e^{-ikL} h_0''(t) - e^{ikL} g_0''(t) \right]
\nn\\
&\quad
+ ik \left[
e^{-ikL} h_3'(t) - e^{ikL} g_3'(t) \right] 
\Big\}\ ,
\\
B(k, t) 
&= 
\frac{1}{1+\delta k^2}
\, \Big\{
-\left[e^{-ikL} h_3(t) - e^{ikL} g_3(t)\right] + \delta \left[ e^{-ikL} h_1'(t) - e^{ikL} g_1'(t)\right] 
\nn\\
&\quad
+ i\delta k \left[ e^{-ikL} h_0'(t) - e^{ikL} g_0'(t)\right]
\Big\}\ .
\nn
\end{align}
Integrating \eqref{lng-ft} with respect to $t$ and using the initial conditions \eqref{lng-ic}, we obtain what is known in the unified transform terminology as the global relation:
\begin{equation}\label{gr-v}
\widehat {\mathbf{v}}(k, t)
=
e^{-M t} \, \widehat {\mathbf{v}}_0(k)
+
e^{-M t} \int_{\tau=0}^t e^{M \tau} \mathbf{A}(k, \tau) d\tau, 
\quad
k \in \mathbb C \setminus \left\{\pm \tfrac{i}{\sqrt \delta}, \pm \tfrac{i}{\sqrt \beta}\right\}\ ,
\end{equation}
where $\widehat {\mathbf{v}}_0(k) = \left(\widehat \eta_0(k), \widehat u_0(k)\right)^T$.

The vector $\mathbf A$ given by \eqref{A-vec-def} involves eight different boundary values, four at each endpoint of the interval $[-L, L]$. However, our analysis will show that only two of these values at each endpoint can be prescribed as boundary data. 
To see this, we must first write the vector equation \eqref{gr-v} in component form. For this purpose, we diagonalize the matrix $M$ in order to express the exponential $e^{Mt}$ in explicit form. We then have $M = P D P^{-1}$ with
\begin{equation}\label{eval-eq}
D = i\omega 
\left(
\begin{array}{cc}
1 & 0
\\
0 &-1
\end{array}
\right),
\quad
P = \left(
\begin{array}{cc}
1 & 1
\\
\frac{\mu_\beta }{\mu_\alpha \mu_\delta  } 
& -\frac{\mu_\beta }{\mu_\alpha \mu_\delta  } 
\end{array}
\right),
\quad
\omega(k) =   \frac{k \mu_\alpha }{\mu_\delta   \mu_\beta }\ ,
\end{equation}
where the functions 
$$
\mu_\alpha(k)  = \left(1+\alpha k^2\right)^{\frac 12},
\quad 
\mu_\delta(k)    = \left(1+\delta k^2\right)^{\frac 12},
\quad
\mu_\beta(k)  = \left(1+\beta k^2\right)^{\frac 12},
$$
are made single-valued by taking appropriate branch cuts in the complex $k$-plane (see Figure \ref{branch-cut-f}). For $\alpha >  \delta > \beta$, these individual branch cuts 
combine to the following branch cut for the function $\omega$:
\begin{equation}\label{bcut-def}
\mathcal B 
:=
i\left[\tfrac{1}{\sqrt\alpha}, \tfrac{1}{\sqrt \delta} \right]\cup \, i\left[-\tfrac{1}{\sqrt \delta}, -\tfrac{1}{\sqrt\alpha}\right]
\cup \, i\left(-\infty, -\tfrac{1}{\sqrt\beta}\right] 
\cup i\left[\tfrac{1}{\sqrt \beta}, \infty\right).
\end{equation}
Also, if $\delta>\alpha>\beta$ then the branch cut for $\omega$ is like the one of Figure \ref{branch-cut-f}  but with $\pm i/\sqrt \alpha$ and $\pm i/\sqrt \delta$ swapped. Therefore, although the branch cut \eqref{bcut-def} corresponds to the case $\alpha>\delta>\beta$, our analysis also covers the case $\delta>\alpha>\beta$.
On the other hand, if $\beta>\max\left\{\alpha, \delta\right\}$ then the branch cut would need to be modified more substantially; however, this case is not really relevant for our purposes, since our main objective behind considering the generalized Nwogu system is to eventually be able to take the limit $\beta \to 0$ and infer results for the original  system (e.g. see Remark \ref{non-reg-r}).
Note that for $\alpha = \delta$ and $\beta = 0$, we have $\omega(k) = k$ and the branching disappears.

\begin{figure}[ht!]
\centering
\begin{minipage}{0.3\textwidth}
\begin{tikzpicture}[scale=2.2, rotate=0]
\draw [->,>=Stealth] (-0.55, 0) -- (1.18, 0);
\draw []  (0, -1.2) -- (0, 1.2);
\draw [->,>=Stealth, dotted]  (0, 1.25) -- (0, 1.4);
\draw [dotted]  (0, -1.25) -- (0, -1.35);
\node[] at (1.02, 0.43) {\fontsize{8}{8} $k$};
\filldraw (0.95, 0.42) circle (0.75pt);
\draw[line width=0.6mm] (0, -0.55) -- (0, -0.2);
\draw[line width=0.6mm] (0, 0.2) -- (0, 0.55);
\draw[dashed, rotate around={-19:(0,0.55)}] (0, 0.55) -- (0.95,0.75);
\draw[dashed, rotate around={-39:(0,-0.55)}] (0, -0.55) -- (0.13,0.75);
\draw [domain=-90:-9, ->, >=stealth] plot ({0.1*cos(\x)}, {0.55+0.1*sin(\x)});
\node[] at (0.12, 0.43) {\fontsize{8}{8} $\phi_2$};
\draw [domain=-90:46, ->, >=stealth] plot ({0.1*cos(\x)}, {-0.55+0.1*sin(\x)});
\node[] at (0.165, -0.47) {\fontsize{8}{8} $\phi_1$};
\node[] at (-0.2, 0.55) {\fontsize{13}{13} $\frac{i}{\sqrt \delta}$};
\node[] at (-0.27, -0.55) {\fontsize{13}{13} $-\frac{i}{\sqrt \delta}$};
\filldraw (0,0.55) circle (0.75pt);
\filldraw (0,-0.55) circle (0.75pt);
\node[] at (-0.21, 0.2) {\fontsize{13}{13} $\frac{i}{\sqrt \alpha}$};
\node[] at (-0.27, -0.2) {\fontsize{13}{13} $-\frac{i}{\sqrt \alpha}$};
\filldraw (0,0.2) circle (0.75pt);
\filldraw (0,-0.2) circle (0.75pt);
\draw[dashed, rotate around={-18:(0,0.2)}] (0, 0.2) -- (0.85,0.71);
\draw[dashed, rotate around={-41.5:(0,-0.2)}] (0, -0.2) -- (0.3,0.9);
\draw [domain=-90:13, ->, >=stealth] plot ({0.1*cos(\x)}, {0.2+0.1*sin(\x)});
\node[] at (0.12, 0.1) {\fontsize{8}{8} $\theta_2$};
\draw [domain=-90:35, ->, >=stealth] plot ({0.1*cos(\x)}, {-0.2+0.1*sin(\x)});
\node[] at (0.16, -0.21) {\fontsize{8}{8} $\theta_1$};
\draw[line width=0.6mm] (0, -1.2) -- (0, -0.85);
\draw[line width=0.6mm] (0, 0.85) -- (0, 1.2);
\node[] at (-0.21, 0.85) {\fontsize{13}{13} $\frac{i}{\sqrt \beta}$};
\node[] at (-0.27, -0.85) {\fontsize{13}{13} $-\frac{i}{\sqrt \beta}$};
\filldraw (0,0.85) circle (0.75pt);
\filldraw (0,-0.85) circle (0.75pt);
\draw[dashed, rotate around={-18:(0,0.85)}] (0, 0.85) -- (0.9,0.75);
\draw[dashed, rotate around={-26.5:(0,-0.85)}] (0, -0.85) -- (0.3,0.7);
\draw [domain=90:-22, ->, >=stealth] plot ({0.1*cos(\x)}, {0.85+0.1*sin(\x)});
\node[] at (0.13, 0.95) {\fontsize{8}{8} $\varphi_2$};
\draw [domain=-90:49, ->, >=stealth] plot ({0.1*cos(\x)}, {-0.85+0.1*sin(\x)});
\node[] at (0.17, -0.85) {\fontsize{8}{8} $\varphi_1$};
\end{tikzpicture}
\end{minipage}
\hspace*{-0.85cm}
\begin{minipage}{0.5\textwidth}
\begin{tikzpicture}[scale=1.7, rotate=0]
\draw [->,>=Stealth] (-2.5, 0) -- (2.65, 0);
\draw [->,>=Stealth] (0,-0.8) -- (0,0.8);
\draw[line width=0.2mm, red] (-2,0.022) -- (-1.75,0.022);
\draw [domain=180:0, ->, >=stealth, red, line width=0.2mm] plot ({-1.5+0.25*cos(\x)}, {0.014+0.25*sin(\x)});
\draw[line width=0.2mm, red] (-1.25,0.022) -- (-1,0.022);
\draw [domain=180:0, ->, >=stealth, red, line width=0.2mm] plot ({-0.75+0.25*cos(\x)}, {0.014+0.25*sin(\x)});
\draw[line width=0.2mm, red] (-0.5,0.022) -- (-0.25,0.022);
\draw [domain=180:0, ->, >=stealth, red, line width=0.2mm] plot ({-0+0.25*cos(\x)}, {0.014+0.25*sin(\x)});
\draw[line width=0.2mm, red] (0.25,0.022) -- (0.5,0.022);
\draw [domain=180:0, ->, >=stealth, red, line width=0.2mm] plot ({0.75+0.25*cos(\x)}, {0.014+0.25*sin(\x)});
\draw[line width=0.2mm, red] (1,0.022) -- (1.25,0.022);
\draw [domain=180:0, ->, >=stealth, red, line width=0.2mm] plot ({1.5+0.25*cos(\x)}, {0.014+0.25*sin(\x)});
\draw[line width=0.2mm, red] (1.75,0.022) -- (2,0.022);
\node[] at (-2.3, 0.15) {$\textcolor{red}{\ldots}$};
\node[] at (2.35, 0.15) {$\textcolor{red}{\ldots}$};
\node[] at (-1.8, 0.43) {\fontsize{9}{9} $\textcolor{red}{\mathcal L^+}$};
\draw[line width=0.2mm, blue] (-2,-0.022) -- (-1.75,-0.022);
\draw [domain=180:360, ->, >=stealth, blue, line width=0.2mm] plot ({-1.5+0.25*cos(\x)}, {-0.014+0.25*sin(\x)});
\draw[line width=0.2mm, blue] (-1.25,-0.022) -- (-1,-0.022);
\draw [domain=180:360, ->, >=stealth, blue, line width=0.2mm] plot ({-0.75+0.25*cos(\x)}, {-0.014+0.25*sin(\x)});
\draw[line width=0.2mm, blue] (-0.5,-0.022) -- (-0.25,-0.022);
\draw [domain=180:360, ->, >=stealth, blue, line width=0.2mm] plot ({-0+0.25*cos(\x)}, {-0.014+0.25*sin(\x)});
\draw[line width=0.2mm, blue] (0.25,-0.022) -- (0.5,-0.022);
\draw [domain=180:360, ->, >=stealth, blue, line width=0.2mm] plot ({0.75+0.25*cos(\x)}, {-0.014+0.25*sin(\x)});
\draw[line width=0.2mm, blue] (1,-0.022) -- (1.25,-0.022);
\draw [domain=180:360, ->, >=stealth, blue, line width=0.2mm] plot ({1.5+0.25*cos(\x)}, {-0.014+0.25*sin(\x)});
\draw[line width=0.2mm, blue] (1.75,-0.022) -- (2,-0.022);
\node[] at (-2.3, -0.15) {$\textcolor{blue}{\ldots}$};
\node[] at (2.35, -0.15) {$\textcolor{blue}{\ldots}$};
\node[] at (-1.8, -0.43) {\fontsize{9}{9} $\textcolor{blue}{\mathcal L^-}$};
\filldraw (-1.5,0) circle (1.3pt);
\filldraw (-0.75,0) circle (1.3pt);
\filldraw (0,0) circle (1.3pt);
\filldraw (0.75,0) circle (1.3pt);
\filldraw (1.5,0) circle (1.3pt);
\end{tikzpicture}
\end{minipage}
%
\caption{\textit{Left:} The branch cut $\mathcal B$ given by \eqref{bcut-def} for the function $\omega$ defined in \eqref{eval-eq}. The local angles $\theta_1, \theta_2, \phi_1, \phi_2, \varphi_1, \varphi_2 \in [0, 2\pi]$ are used to make each of the square roots $\mu_\alpha$, $\mu_\delta$ and $\mu_\beta$ single-valued by taking branch cuts along $i\left[-1/\sqrt \alpha, 1/\sqrt \alpha\right]$,  $i\left[-1/\sqrt \delta, 1/\sqrt\delta\right]$ and $i\left(-\infty, -1/\sqrt\beta\right] \cup i\left[1/\sqrt \beta, \infty\right)$ respectively.
\textit{Right:} The contours $\mathcal L^\pm$ defined by \eqref{lpm-def}. The black dots correspond to the simple zeros $k_n = \frac{n\pi}{2L}$, $n\in\mathbb Z$, of the quantity $\Delta$ involved in Proposition~\ref{jordan-p}.
}
\label{branch-cut-f}
\end{figure}
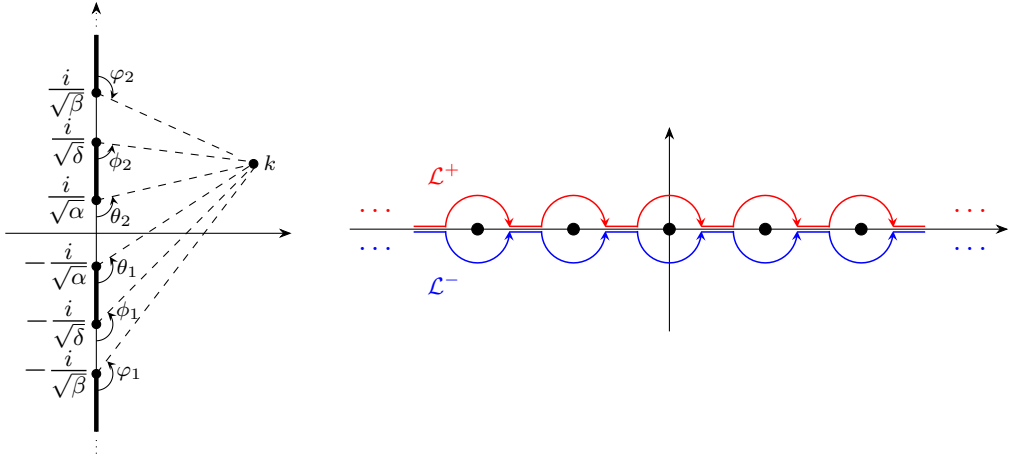

Using the diagonalization of $M$, we obtain 
$e^{-Mt} = P e^{-Dt} P^{-1}$. Then, 
introducing the notation
\begin{equation}\label{bpm-def}
B_f^\pm(\omega, t) 
:= 
e^{i\omega t}  \int_{\tau=0}^t e^{-i \omega \tau} f(\tau) d\tau 
\pm 
e^{-i\omega t} \int_{\tau=0}^t e^{i \omega \tau} f(\tau) d\tau\ ,
\end{equation}
we write the global relation \eqref{gr-v} in component form as follows:
\begin{subequations}\label{gr-comb}
\begin{align}\label{eta-gr}
\widehat  \eta(k, t)
&=
\frac 12 \left(e^{i\omega t} + e^{-i\omega t}\right) \widehat  \eta_0(k)
-
 \frac{\mu_\alpha \mu_\delta  }{2\mu_\beta }\left(e^{i\omega t}-e^{-i\omega t}\right) \widehat  u_0(k)
\nn\\
&
+
\frac{1+\alpha k^2}{2 \left(1+\beta k^2\right)}   \left[
-e^{-ikL} B_{h_0}^+(\omega, t)
+ e^{ikL} B_{g_0}^+(\omega, t) \right]
+
\frac{i\alpha k}{2 \left(1+\beta k^2\right)}  \left[
e^{-ikL} B_{h_1}^+(\omega, t)
- e^{ikL} B_{g_1}^+(\omega, t) \right]
\nn\\
&
+
\frac{\alpha}{2 \left(1+\beta k^2\right)}   \left[
e^{-ikL} B_{h_2}^+(\omega, t)
- e^{ikL} B_{g_2}^+(\omega, t)\right]
+
\frac{\beta \delta}{2\left(1+\beta k^2\right)} \left[
e^{-ikL} B_{h_2''}^+(\omega, t)
-
e^{ikL}B_{g_2''}^+(\omega, t)
\right]
\nn\\
&
+
\frac{\beta}{2\left(1+\beta k^2\right)} \left[
-e^{-ikL} B_{h_0''}^+(\omega, t)
+
e^{ikL} B_{g_0''}^+(\omega, t)
\right]
+
\frac{i \beta k}{2\left(1+\beta k^2\right)} \left[
e^{-ikL} B_{h_3'}^+(\omega, t)
-
e^{ikL} B_{g_3'}^+(\omega, t)
\right]
\nn\\
&
+  \frac{\mu_\alpha }{2\mu_\delta   \mu_\beta } 
\left[
e^{-ikL} B_{h_3}^-(\omega, t)
-
e^{ikL} B_{g_3}^-(\omega, t)
\right]
+  
\frac{\delta \mu_\alpha }{2 \mu_\delta   \mu_\beta }
\left[ 
-e^{-ikL} B_{h_1'}^-(\omega, t)
+
e^{ikL}
B_{g_1'}^-(\omega, t)
\right]
\nn\\
&
+ \frac{i\delta k \mu_\alpha }{2\mu_\delta   \mu_\beta } 
\left[ 
-e^{-ikL} B_{h_0'}^-(\omega, t)
+
e^{ikL} B_{g_0'}^-(\omega, t)
\right], \quad k \in \mathbb C \setminus \mathcal B\ , 
\end{align}
and
\begin{align}\label{u-gr}
\widehat  u(k, t)
&=
-\frac{\mu_\beta }{2 \mu_\alpha \mu_\delta  } \left(e^{i\omega t}-e^{-i\omega t}\right) \widehat  \eta_0(k)
+
\frac 12 \left(e^{i\omega t} + e^{-i\omega t}\right) \widehat  u_0(k)
\nn\\
&
+\frac{\mu_\alpha }{2\mu_\delta   \mu_\beta } 
\left[
e^{-ikL} B_{h_0}^-(\omega, t)
-
e^{ikL} B_{g_0}^-(\omega, t)
\right]
+
\frac{i\alpha k}{2 \mu_\alpha \mu_\delta   \mu_\beta } 
\left[
-e^{-ikL} B_{h_1}^-(\omega, t)
+
e^{ikL} B_{g_1}^-(\omega, t)
\right]
\nn\\
&
+\frac{\alpha}{2 \mu_\alpha \mu_\delta   \mu_\beta } 
\left[
-e^{-ikL} B_{h_2}^-(\omega, t)
+
e^{ikL} B_{g_2}^-(\omega, t)
\right]
+
\frac{\beta \delta}{2\mu_\alpha \mu_\delta  \mu_\beta } \left[
-e^{-ikL} B_{h_2''}^-(\omega, t)
+
e^{ikL} B_{g_2''}^-(\omega, t)
\right]
\nn\\
&
+
\frac{\beta}{2\mu_\alpha \mu_\delta  \mu_\beta } 
\left[
e^{-ikL} B_{h_0''}^-(\omega, t)
-
e^{ikL} B_{g_0''}^-(\omega, t)
\right]
+
\frac{i \beta k}{2\mu_\alpha \mu_\delta  \mu_\beta } \left[
-e^{-ikL} B_{h_3'}^-(\omega, t)
+
e^{ikL} B_{g_3'}^-(\omega, t)
\right]
\nn\\
&
+\frac{1}{2\left(1+\delta k^2\right)}
\left[
-e^{-ikL} B_{h_3}^+(\omega, t)
+
e^{ikL} B_{g_3}^+(\omega, t)
\right]
+
\frac{\delta}{2\left(1+\delta k^2\right)}
\left[
e^{-ikL} B_{h_1'}^+(\omega, t)
-
e^{ikL}B_{g_1'}^+(\omega, t)
\right]
\nn\\
&
+\frac{i\delta k}{2\left(1+\delta k^2\right)}
\left[
e^{-ikL} B_{h_0'}^+(\omega, t)
-
e^{ikL} B_{g_0'}^+(\omega, t)
\right],\quad k \in \mathbb C \setminus \mathcal B\ .
\end{align}
\end{subequations}
Inverting   \eqref{eta-gr} for $k\in \mathbb R$ by means of \eqref{ft-def} yields the following integral representation for $\eta$:
\begin{subequations}\label{ir-v}
\begin{align}\label{eta-ir}
2\pi \eta(x, t)
&=
 \int_{k\in\mathbb R} e^{ikx} \, \frac 12 \left[\left(e^{i\omega t} + e^{-i\omega t}\right) \widehat  \eta_0(k)  
-
 \frac{\mu_\alpha \mu_\delta  }{\mu_\beta }
  \left(e^{i\omega t}-e^{-i\omega t}\right) \widehat  u_0(k) \right] dk
\nn\\
&
+
 \int_{k\in\mathbb R} e^{ik(x-L)} \,
\bigg[
- \frac{1+\alpha k^2}{2 \left(1+\beta k^2\right)}  \, B_{h_0}^+(\omega, t) 
+
\frac{\alpha}{2 \left(1+\beta k^2\right)} \, B_{h_2}^+(\omega, t) 
\nn\\
&\quad
+
\frac{\beta \delta}{2 \left(1+\beta k^2\right)}  \, B_{h_2''}^+(\omega, t)
-
\frac{\beta}{2 \left(1+\beta k^2\right)} \, B_{h_0''}^+(\omega, t)
-
\frac{i\delta k\mu_\alpha }{2\mu_\delta   \mu_\beta }
\, B_{h_0'}^-(\omega, t)
\bigg]
dk
\nn\\
&
-
 \int_{k\in\mathbb R} e^{ik(x+L)} \,
\bigg[
-\frac{1+\alpha k^2}{2 \left(1+\beta k^2\right)} \, B_{g_0}^+(\omega, t) +
\frac{\alpha}{2 \left(1+\beta k^2\right)} \, B_{g_2}^+(\omega, t)
\nn\\
&\quad
+
\frac{\beta \delta}{2 \left(1+\beta k^2\right)} \, B_{g_2''}^+(\omega, t) 
-
\frac{\beta}{2 \left(1+\beta k^2\right)} \, B_{g_0''}^+(\omega, t)
-
 \frac{i\delta k\mu_\alpha }{2\mu_\delta   \mu_\beta }\, B_{g_0'}^-(\omega, t)
\bigg] dk
\nn\\
&
+
 \int_{k\in\mathbb R} e^{ik(x-L)} \,
 \bigg[
\frac{i \alpha k}{2 \left(1+\beta k^2\right)} \, B_{h_1}^+(\omega, t)
+
\frac{i\beta k}{2 \left(1+\beta k^2\right)} \, B_{h_3'}^+(\omega, t)
\nn\\
&\hskip 2.9cm
+
\frac{\mu_\alpha }{2\mu_\delta   \mu_\beta } \, B_{h_3}^-(\omega, t)
-
 \frac{\delta \mu_\alpha }{2\mu_\delta   \mu_\beta } 
\, B_{h_1'}^-(\omega, t)
\bigg]
 dk
\nn\\
&
-
 \int_{k\in\mathbb R} e^{ik(x+L)}  \,
\bigg[
\frac{i \alpha k}{2 \left(1+\beta k^2\right)} \, B_{g_1}^+(\omega, t)
+
\frac{i\beta k}{2 \left(1+\beta k^2\right)} \, B_{g_3'}^+(\omega, t) 
\nn\\
&\hskip 2.9cm
+
\frac{\mu_\alpha }{2\mu_\delta   \mu_\beta } \, B_{g_3}^-(\omega, t)
-
\frac{\delta \mu_\alpha }{2\mu_\delta   \mu_\beta }   \, B_{g_1'}^-(\omega, t)
\bigg] dk\ .
\end{align}
Similarly, inverting \eqref{u-gr} for $k\in\mathbb R$, we obtain the following integral representation for $u$:
\begin{align}\label{u-ir}
2\pi  u(x, t)
&=
 \int_{k\in\mathbb R} e^{ikx} \, \frac 12 \left[
-\frac{\mu_\beta }{\mu_\alpha \mu_\delta  } \left(e^{i\omega t}-e^{-i\omega t}\right) \widehat  \eta_0(k)
+
 \left(e^{i\omega t} + e^{-i\omega t}\right) \widehat  u_0(k)
 \right] dk
\nn\\
&
+
 \int_{k\in\mathbb R} e^{ik(x-L)} \,
 \bigg[
\frac{\mu_\alpha }{2\mu_\delta   \mu_\beta } 
\, B_{h_0}^-(\omega, t)   
-
\frac{\alpha}{2 \mu_\alpha \mu_\delta   \mu_\beta } 
\, B_{h_2}^-(\omega, t) 
\nn\\
&\quad
-
\frac{\beta \delta}{2 \mu_\alpha \mu_\delta   \mu_\beta } 
\, B_{h_2''}^-(\omega, t)
+
\frac{\beta}{2 \mu_\alpha \mu_\delta   \mu_\beta } 
\, B_{h_0''}^-(\omega, t)  
+
\frac{i\delta k}{2\left(1+\delta k^2\right)}
\, B_{h_0'}^+(\omega, t) 
\bigg] dk
\nn\\
&-
 \int_{k\in\mathbb R} e^{ik(x+L)} \, 
 \bigg[
 \frac{\mu_\alpha }{2\mu_\delta   \mu_\beta } \, B_{g_0}^-(\omega, t)  -
 \frac{\alpha}{2 \mu_\alpha \mu_\delta   \mu_\beta } \, B_{g_2}^-(\omega, t) 
\nn\\
&\quad
-
\frac{\beta \delta}{2 \mu_\alpha \mu_\delta   \mu_\beta } \, B_{g_2''}^-(\omega, t)  
+
\frac{\beta}{2 \mu_\alpha \mu_\delta   \mu_\beta } \, B_{g_0''}^-(\omega, t)  
+
\frac{i\delta k}{2\left(1+\delta k^2\right)}
\, B_{g_0'}^+(\omega, t)
\bigg]  dk
\nn\\
&
+
 \int_{k\in\mathbb R} e^{ik(x-L)} \, 
 \bigg[
-\frac{i\alpha k}{2 \mu_\alpha \mu_\delta   \mu_\beta } 
\, B_{h_1}^-(\omega, t)  
-
\frac{i \beta k}{2 \mu_\alpha \mu_\delta   \mu_\beta } 
\, B_{h_3'}^-(\omega, t) 
\nn\\
&\hskip 2.9cm
-
\frac{1}{2\left(1+\delta k^2\right)}
\, B_{h_3}^+(\omega, t)
+
\frac{\delta}{2\left(1+\delta k^2\right)}
\, B_{h_1'}^+(\omega, t)
\bigg] dk
\nn\\
&
-
 \int_{k\in\mathbb R} e^{ik(x+L)} \,
 \bigg[
-\frac{i\alpha k}{2 \mu_\alpha \mu_\delta   \mu_\beta } \, B_{g_1}^-(\omega, t) 
-
\frac{i \beta k}{2 \mu_\alpha \mu_\delta   \mu_\beta } \, B_{g_3'}^-(\omega, t) 
\nn\\
&\hskip 2.9cm
-
\frac{1}{2\left(1+\delta k^2\right)}
\, B_{g_3}^+(\omega, t) 
+
\frac{\delta}{2\left(1+\delta k^2\right)} 
\, B_{g_1'}^+(\omega, t)  
\bigg] dk\ .
\end{align}
\end{subequations}

\subsection{Elimination of boundary values}\label{ebv-ss}
The integral representations \eqref{ir-v} involve eight boundary values, four at each endpoint of $[-L, L]$, through the transforms $B_{g_j}^\pm$ and $B_{h_j}^\pm$, $j=0,1,2,3$. However, our analysis allows us to eliminate four of these boundary values, two from each endpoint. 

The key idea lies in the observation that $\omega(-k) = -\omega(k)$, since the transformation $k\mapsto -k$ flips the sign of  $\mu_\alpha$ and  $\mu_\delta$   but leaves $\mu_{\beta}$  invariant (see local angles in Figure~\ref{branch-cut-f}).
Then, recalling definition \eqref{bpm-def}, we deduce  
$B_f^+(-\omega, t) = B_f^+(\omega, t)$
and
$B_f^-(-\omega, t) = -B_f^-(\omega, t)$.
Therefore, applying the transformation $k\mapsto -k$ to the global relations \eqref{gr-comb} generates \textit{two additional identities} which are also valid in $\mathbb C \setminus \mathcal B$. 
These new identities together with \eqref{gr-comb} can be combined to yield expressions for the quantities 
$\widehat \eta(k, t) + e^{\pm 2ikL} \, \widehat \eta(-k, t)$ and $\widehat u(k, t) - e^{\pm 2ikL} \, \widehat u(-k, t)$, respectively. 

In particular, the expression emerging for the quantity $\widehat \eta(k, t) + e^{-2ikL} \, \widehat \eta(-k, t)$ can be used in order to eliminate from the integral representation \eqref{eta-ir} for $\eta$ the combination of terms that involve the boundary values $g_1$ and $g_3$ (namely, the last integral in \eqref{eta-ir}) in favor of a combination that involves the boundary values $g_0, h_0, g_2, h_2$ (along with other, non-boundary-value terms). 
In addition, the expression for the quantity $\widehat \eta(k, t) + e^{2ikL} \, \widehat \eta(-k, t)$ can be used  for the elimination from \eqref{eta-ir} of the combination of terms that involve the boundary values $h_1$ and $h_3$ (i.e. the penultimate integral in \eqref{eta-ir}) in favor of a combination that involves once again the boundary values $g_0, h_0, g_2, h_2$ (along with other, non-boundary-value terms).
Similarly, the expression for  the quantity $\widehat u(k, t) - e^{-2ikL} \, \widehat u(-k, t)$ can be used  in order to eliminate from the integral representation \eqref{u-ir} for $u$ the combination of terms that involve the boundary values $g_1$ and $g_3$ (namely, the last integral in \eqref{u-ir}) in favor of a combination that involves  the boundary values $g_0, h_0, g_2, h_2$ (along with other, non-boundary-value terms), 
while the expression for the quantity $\widehat u(k, t) - e^{2ikL} \, \widehat u(-k, t)$ can be used for the elimination from \eqref{u-ir} of the combination of terms in \eqref{ir-v} that involve the boundary values $h_1$ and $h_3$ (i.e. the penultimate integral in \eqref{u-ir}) in favor of a combination that involves the boundary values $g_0, h_0, g_2, h_2$ (along with other, non-boundary-value terms).

\begin{remark}[Admissible boundary conditions for the Nwogu system]
\label{adm-bc-r}
The manipulations described above can be performed in order to eliminate the terms in \eqref{ir-v} that involve the boundary values $g_0, h_0, g_2, h_2$ in favor of those that involve $g_1, h_1, g_3, h_3$. Simply, instead of combining the global relations \eqref{gr-v} and their $k\mapsto -k$ counterparts towards obtaining expressions for the quantities $\widehat \eta(k, t) + e^{\pm 2ikL} \, \widehat \eta(-k, t)$ and $\widehat u(k, t) - e^{\pm 2ikL} \, \widehat u(-k, t)$, we  form the combinations  $\widehat \eta(k, t) - e^{\pm 2ikL} \, \widehat \eta(-k, t)$ and $\widehat u(k, t) + e^{\pm 2ikL} \, \widehat u(-k, t)$ which can then be employed for the elimination of the second and third integrals in each of the integral representations \eqref{ir-v} in favor of the boundary values $g_1, h_1, g_3, h_3$ (along with other, non-boundary-value terms).
Therefore, we overall conclude that the two combinations of boundary conditions given in \eqref{adm-bc-i}, namely either $\left\{u(\pm L, t), u_{xx}(\pm L, t)\right\}$ or $\left\{\eta(\pm L, t), u_x(\pm L, t)\right\}$,  both  lead to a well-formulated initial-boundary value problem for the linear generalized Nwogu system \eqref{lrn-i} and hence for the nonlinear generalized Nwogu system \eqref{lng-ibvp-i}, including the regularized system \eqref{eq:Nwogu3nd} and, importantly,  the original Nwogu system.
Furthermore, linear combinations of the aforementioned pairs of boundary values in the form of Robin boundary conditions could also be prescribed.
\end{remark}

If the expressions for the unknown boundary-value quantities  obtained via the calculations described above are introduced directly in the integral representations \eqref{ir-v}, then these representations will degenerate to tautologies.
For this reason, before performing the above calculations we employ Cauchy's integral theorem to \textit{deform} the contours of integration of the boundary-value terms in~\eqref{ir-v} to the contours $\mathcal L^+$ (for the terms involving $g_j$'s) and  $\mathcal L^-$ (for the terms involving $h_j$'s), which are depicted in Figure \ref{branch-cut-f} and are defined by 
\begin{equation}\label{lpm-def}
\mathcal L^+ = 
\widetilde{\mathbb R}  \cup \bigcup_{n\in\mathbb Z} - C_{\frac{\pi}{6L}, [0, \pi]}(k_n),
\quad
\mathcal L^- = 
\widetilde{\mathbb R}  \cup \bigcup_{n\in\mathbb Z} C_{\frac{\pi}{6L}, [\pi, 2\pi]}(k_n)\ ,
\end{equation}
where for $k_n = \frac{n\pi}{2L}$, $n \in \mathbb Z$, we define
$$
\widetilde{\mathbb R} = \bigcup_{n\in\mathbb Z} \left[ k_n + \frac{\pi}{6L}, k_{n+1} - \frac{\pi}{6L}\right],
\quad
C_{r, [a, b]}(k_n) = \Big\{ \left|k-k_n\right| = r, \ a \leq \arg(k) \leq b \Big\}\ .
$$
We emphasize that the deformations from $\mathbb R$ to $\mathcal L^\pm$ are allowed thanks to the analyticity of the relevant integrands in \eqref{ir-v} and  the exponential decay of the terms $e^{ik(x+L)}$ and $e^{ik(x-L)}$ in the upper and lower half of the complex $k$-plane, respectively. These deformations allow us to handle the unknown terms $\widehat \eta(\pm k, t)$ and $\widehat u(\pm k, t)$ that arise after the calculations described above via the following result.
\begin{proposition}\label{jordan-p}
Let $\Delta(k) = e^{ikL} - e^{-3ikL}$ and suppose  $f \in H^1(-L, L)$. Then, for all $x \in (-L, L)$,
%
\begin{align*}
&\int_{k\in \mathcal L^+} \frac{e^{ik(x+L)}}{\Delta(k)} \,  \widehat f(k) dk
=
0
=
\int_{k\in \mathcal L^+} \frac{e^{ik(x+L)}}{\Delta(k)} \,  e^{2ikL} \, \widehat f(-k) dk\ ,
\\
&\int_{k\in \mathcal L^-} \frac{e^{ik(x-L)}}{\overline{\Delta(\bar k)}} \,   \widehat f(k)   dk
=
0
=
\int_{k\in \mathcal L^-} \frac{e^{ik(x-L)}}{\overline{\Delta(\bar k)}} \,  e^{-2ikL} \, \widehat f(-k)   dk\ .
\end{align*}
%
\end{proposition}

\begin{proof}
We only give the proof for the first of the above integrals as the remaining ones can be handled in an entirely analogous way.
Since $\widehat f(k)$ is entire and $\Delta(k) \neq 0$ away from $\mathbb R$, employing Cauchy's theorem we can write
\begin{equation*}
\int_{k\in \mathcal L^+} \frac{e^{ik(x+L)}}{\Delta(k)} \,  \widehat f(k) dk
=
\lim_{R\to \infty}
\int_{k\in C_{R, \theta_0}} \frac{e^{ik(x+L)}}{\Delta(k)} \,  \widehat f(k) dk\ ,
\end{equation*}
where for the circular arc
$
C_{R, \theta_0}
=
\left\{R e^{i\theta}, \ \theta_0 \leq \theta \leq \pi - \theta_0 \right\}
$
we take $\theta_0 = 0$ if $\left|R-k_n\right|\geq \frac{\pi}{6L}$ for all $n\in\mathbb Z$
and $0 < \theta_0 \leq \sin^{-1}\left(\frac{\pi}{6LR}\right)$ if there exists $n\in \mathbb Z$ such that $\left|R-k_n\right| < \frac{\pi}{6L}$.
Thus, integrating by parts and substituting for $\Delta(k)$, we have
$$
\int_{k\in C_{R, \theta_0}} \frac{e^{ik(x+L)}}{\Delta(k)} \,  \widehat f(k) dk
=
\int_{k\in C_{R, \theta_0}} \frac{e^{ik(x+4L)}}{e^{4ikL}-1}
\bigg\{
\left[ e^{ikL} f(-L) - e^{-ikL} f(L) \right] 
+  \int_{y=-L}^L e^{-iky} f'(y) dy
\bigg\}
\frac{1}{ik} \, dk\ .
$$
Next, we use the elementary lower bound (see e.g.  \cite{wm2021} for a proof)
$$
\left|e^{4ikL}-1\right| \geq 1-e^{-\frac{2\pi}{3}}, \quad 
k\in \left\{\textnormal{Im}(k)\geq0\right\}\setminus \bigcup_{n\in\mathbb Z} D_{\frac{\pi}{6L}}(k_n).
$$
Since $C_{R, \theta_0} \subset \left\{\textnormal{Im}(k)\geq0\right\}\setminus \bigcup_{n\in\mathbb Z} D_{\frac{\pi}{6L}}(k_n)$, employing  the above bound we infer
\begin{align*}
\left|
\int_{k\in C_{R, \theta_0}} \frac{e^{ik(x+L)}}{\Delta(k)} \,  \widehat f(k) dk
\right|
&\leq
2\left|f(-L)\right|
\int_{\theta=\theta_0}^{\frac \pi 2}
 \frac{e^{-R(x+5L)\sin\theta}}{1-e^{-\frac{2\pi}{3}}}
\,
 d\theta
+
2\left|f(L)\right|
\int_{\theta=\theta_0}^{\frac \pi 2}
 \frac{e^{-R(x+3L)\sin\theta}}{1-e^{-\frac{2\pi}{3}}}
\,    d\theta
\nn\\
&\quad
+
2\int_{\theta=\theta_0}^{\frac \pi 2}
 \frac{e^{-R(x+4L)\sin\theta}}{1-e^{-\frac{2\pi}{3}}}
\int_{y=-L}^L e^{Ry\sin\theta} \left|f'(y)\right| dy
\,
d\theta
=:
I_1 + I_2 + I_3\ .
\end{align*}
Using the well-known inequality $\sin\theta \geq 2\theta/\pi$, $0\leq \theta \leq  \pi/2$, we have
$$
I_1 
\leq
\frac{2\left|f(-L)\right|}{1-e^{-\frac{2\pi}{3}}}
\int_{\theta=\theta_0}^{\frac \pi 2}
e^{-\frac{2R(x+5L)}{\pi} \, \theta}
\,
 d\theta
=
\frac{\pi\left|f(-L)\right|}{\big(1-e^{-\frac{2\pi}{3}}\big)R\left(x+5L\right)} 
\left[ e^{-\frac{2R(x+5L)}{\pi} \, \theta_0} - e^{-R(x+5L)} \right]
\overset{R\to\infty}{\xrightarrow{\hspace*{7mm}}} 0\ ,
$$
and, similarly, 
$$
I_2
\leq
\frac{2\left|f(L)\right|}{1-e^{-\frac{2\pi}{3}}}
\int_{\theta=\theta_0}^{\frac \pi 2}
e^{-\frac{2R(x+3L)}{\pi} \, \theta}
\,
 d\theta
=
\frac{\pi\left|f(L)\right|}{\big(1-e^{-\frac{2\pi}{3}}\big)R\left(x+3L\right)} 
\left[ e^{-\frac{2R(x+3L)}{\pi} \, \theta_0} - e^{-R(x+3L)} \right]
\overset{R\to\infty}{\xrightarrow{\hspace*{7mm}}} 0\ .
$$
Furthermore, using in addition Fubini's theorem, we find
\begin{align*}
I_3
&\leq
\frac{2}{1-e^{-\frac{2\pi}{3}}} \left\| f' \right\|_{L^1(-L, L)}
\int_{\theta=\theta_0}^{\frac \pi 2}
e^{-\frac{2R(x+3L)}{\pi} \,  \theta} 
 d\theta  
\nn\\
&=
\frac{\pi \left\| f' \right\|_{L^1(-L, L)}}{\big(1-e^{-\frac{2\pi}{3}}\big)R\left(x+3L\right)} 
\left[ e^{-\frac{2R(x+3L)}{\pi} \, \theta_0} - e^{-R(x+3L)} \right]
\overset{R\to\infty}{\xrightarrow{\hspace*{7mm}}} 0\ .
\end{align*}
Hence, overall we conclude that 
$$
\lim_{R\to\infty} 
\int_{k\in \mathcal L^+} \frac{e^{ik(x+L)}}{\Delta(k)} \,  \widehat f(k) dk
=
\lim_{R\to\infty} 
\int_{k\in C_{R, \theta_0}} \frac{e^{ik(x+L)}}{\Delta(k)} \,  \widehat f(k) dk
 = 0, \quad x\in (-L, L)\ ,
$$
as desired, completing the proof of Proposition \ref{jordan-p}.
\end{proof}

\subsection{Solution formulas via the unified transform}
Altogether, performing the calculations described at the beginning of Subsection \ref{ebv-ss} and then using  Proposition \ref{jordan-p}, we obtain 
\begin{subequations}\label{sols-comb}
\begin{align}\label{eta-sol}
\eta(x, t)
&=
\frac{1}{2\pi} \int_{k\in\mathbb R} e^{ikx} \, \frac 12 \left[\left(e^{i\omega t} + e^{-i\omega t}\right) \widehat  \eta_0(k)  
-
 \frac{\mu_\alpha \mu_\delta  }{\mu_\beta }
  \left(e^{i\omega t}-e^{-i\omega t}\right) \widehat  u_0(k) \right] dk
\nn\\
&
+
\frac{1}{2\pi} \int_{k\in\mathcal L^-} 
\frac{e^{ik(x-L)} }{2\left(1-e^{-4ikL}\right)} 
\, \bigg\{ 
\left(e^{i\omega t} + e^{-i\omega t}\right) \left[e^{-3ikL} \, \widehat  \eta_0(k) + e^{-ikL} \, \widehat  \eta_0(-k) \right]
\nn\\
&\quad
-
\frac{\mu_\alpha \mu_\delta  }{\mu_\beta }  \left(e^{i\omega t}-e^{-i\omega t}\right) \left[ e^{-3ikL} \, \widehat  u_0(k) - e^{-ikL} \, \widehat  u_0(-k) \right] \bigg\} \, dk
\nn\\
&
+
\frac{1}{2\pi} \int_{k\in\mathcal L^+}  
\frac{e^{ik(x+L)} }{2\left(1-e^{4ikL}\right)} \, \bigg\{ 
\left(e^{i\omega t} + e^{-i\omega t}\right) \left[e^{3ikL} \, \widehat  \eta_0(k) + e^{ikL} \, \widehat  \eta_0(-k) \right]
\nn\\
&\quad
-
\frac{\mu_\alpha \mu_\delta  }{\mu_\beta }  \left(e^{i\omega t}-e^{-i\omega t}\right) \left[ e^{3ikL} \, \widehat  u_0(k) - e^{ikL} \, \widehat  u_0(-k)\right]
  \bigg\} \, dk
\nn\\
&
+
\frac{1}{2\pi} \int_{k\in\mathcal L^-} \frac{e^{ik(x-L)}}{1-e^{-4ikL}} \,
\bigg[
-\frac{1+\alpha k^2}{1+\beta k^2} \,  B_{h_0}^+(\omega, t) 
+\frac{\alpha}{1+\beta k^2} \, B_{h_2}^+(\omega, t)
\nn\\
&\quad
+\frac{\beta \delta}{1+\beta k^2} \,  B_{h_2''}^+(\omega, t) 
-\frac{\beta}{1+\beta k^2} \,  B_{h_0''}^+(\omega, t) 
-\frac{i\delta k\mu_\alpha }{\mu_\delta   \mu_\beta } \, B_{h_0'}^-(\omega, t) 
\bigg] dk
\nn\\
&
+
\frac{1}{2\pi} \int_{k\in\mathcal L^+} \frac{e^{ik(x+L)}}{1-e^{4ikL}}  \, 
\bigg[
\frac{1+\alpha k^2}{1+\beta k^2} \,  B_{g_0}^+(\omega, t) 
-
\frac{\alpha}{1+\beta k^2} \,  B_{g_2}^+(\omega, t)  
\nn\\
&\quad
- 
\frac{\beta \delta}{1+\beta k^2} \,  B_{g_2''}^+(\omega, t)   
+
\frac{\beta}{1+\beta k^2} \,  B_{g_0''}^+(\omega, t)   
+
 \frac{i\delta k\mu_\alpha }{\mu_\delta   \mu_\beta } \,  B_{g_0'}^-(\omega, t)  
\bigg]  dk
\nn\\
&
+
\frac{1}{2\pi} \int_{k\in\mathcal L^-}  
\frac{e^{ik(x+L)}}{1 - e^{4ikL}} \, \bigg[ 
-\frac{1+\alpha k^2}{1+\beta k^2}
\,  B_{g_0}^+(\omega, t) 
-
 \frac{i\delta k \mu_\alpha }{\mu_\delta   \mu_\beta }
 \, B_{g_0'}^-(\omega, t) 
\nn\\
&\quad
-
\frac{\beta}{1+\beta k^2} \, B_{g_0''}^+(\omega, t) 
+
\frac{\alpha}{1+\beta k^2} \,  B_{g_2}^+(\omega, t)  
+
\frac{\beta \delta}{1+\beta k^2} \, B_{g_2''}^+(\omega, t)  
\bigg]
\, dk
\nn\\
&
+
\frac{1}{2\pi} \int_{k\in\mathcal L^+}   
\frac{e^{ik(x-L)}}{1 - e^{-4ikL}} \, \bigg[ 
\frac{1+\alpha k^2}{1+\beta k^2} \,  B_{h_0}^+(\omega, t)
+
 \frac{i\delta k \mu_\alpha }{\mu_\delta   \mu_\beta }
 B_{h_0'}^-(\omega, t)
\nn\\
&\quad
+
\frac{\beta}{1+\beta k^2}   B_{h_0''}^+(\omega, t)
-
\frac{\alpha}{1+\beta k^2} \, B_{h_2}^+(\omega, t)
- \frac{\beta \delta}{1+\beta k^2} 
 B_{h_2''}^+(\omega, t)
\bigg]
\, dk\ ,
\end{align}
which is an \textit{explicit solution formula} for the  $\eta$-component of the linear  generalized Nwogu system \eqref{lng-ibvp} in the case of an initial-boundary value problem with boundary conditions $u(-L, t) = g_0(t)$, $u(L, t) = h_0(t)$, $u_{xx}(-L, t) = g_2(t)$ and $u_{xx}(L, t) = h_2(t)$. 

For the same boundary conditions, we similarly obtain the following \textit{explicit solution formula} for the  $u$-component of the linear  generalized Nwogu system \eqref{lng-ibvp}:
\begin{align}\label{u-sol}
u(x, t)
&=
\frac{1}{2\pi} \int_{k\in\mathbb R} e^{ikx} \, \frac 12 \left[
-\frac{\mu_\beta }{\mu_\alpha \mu_\delta  } \left(e^{i\omega t}-e^{-i\omega t}\right) \widehat  \eta_0(k)
+
 \left(e^{i\omega t} + e^{-i\omega t}\right) \widehat  u_0(k)
 \right] dk
\nn\\
&
+
\frac{1}{2\pi} \int_{k\in \mathcal L^-}  
\,
\frac{e^{ik(x-L)}}{2\left(1-e^{-4ikL}\right)} 
\,
\bigg\{
-\frac{\mu_\beta }{\mu_\alpha \mu_\delta  } \left(e^{i\omega t}-e^{-i\omega t}\right) \big[ e^{-3ikL} \, \widehat  \eta_0(k) 
\nn\\
&\quad
+ e^{-ikL} \, \widehat  \eta_0(-k)\big]
+
\left(e^{i\omega t} + e^{-i\omega t}\right) \left[ e^{-3ikL} \, \widehat  u_0(k) - e^{-ikL} \, \widehat  u_0(-k)\right]
\bigg\} \, dk
\nn\\
&
+
\frac{1}{2\pi} \int_{k\in \mathcal L^+} \frac{e^{ik(x+L)}}{2\left(1-e^{4ikL}\right)}
\,
\bigg\{
- \frac{\mu_\beta }{\mu_\alpha \mu_\delta  } \left(e^{i\omega t}-e^{-i\omega t}\right) \big[ e^{3ikL}\, \widehat  \eta_0(k) 
\nn\\
&\quad
+ e^{ikL} \, \widehat  \eta_0(-k)\big]
+
\left(e^{i\omega t} + e^{-i\omega t}\right) \left[ e^{3ikL}\,  \widehat  u_0(k) - e^{ikL} \, \widehat  u_0(-k)\right]
\bigg\} \, dk
\nn\\
&
+
\frac{1}{2\pi} \int_{k\in\mathcal L^-} \frac{e^{ik(x-L)}}{1-e^{-4ikL}} \, 
\bigg[
\frac{\mu_\alpha }{\mu_\delta   \mu_\beta } 
\, B_{h_0}^-(\omega, t)  
+
\frac{i\delta k}{1+\delta k^2}
\, B_{h_0'}^+(\omega, t)
\nn\\
&\quad
+
\frac{\beta}{\mu_\alpha \mu_\delta   \mu_\beta } 
\, B_{h_0''}^-(\omega, t)
-
\frac{\alpha}{\mu_\alpha \mu_\delta   \mu_\beta } 
\, B_{h_2}^-(\omega, t)
-
\frac{\beta \delta}{\mu_\alpha \mu_\delta   \mu_\beta } 
\, B_{h_2''}^-(\omega, t) 
\bigg] dk
\nn\\
&
+
\frac{1}{2\pi} \int_{k\in\mathcal L^+} \frac{e^{ik(x+L)}}{1-e^{4ikL}} \, 
\bigg[
-\frac{\mu_\alpha }{2\mu_\delta   \mu_\beta } \, B_{g_0}^-(\omega, t)  
-
\frac{i\delta k}{1+\delta k^2} \, B_{g_0'}^+(\omega, t)  
\nn\\
&\quad
-
\frac{\beta}{\mu_\alpha \mu_\delta   \mu_\beta } \, B_{g_0''}^-(\omega, t)  
+
\frac{\alpha}{\mu_\alpha \mu_\delta   \mu_\beta } \, B_{g_2}^-(\omega, t)  
+
\frac{\beta \delta}{\mu_\alpha \mu_\delta   \mu_\beta } \, B_{g_2''}^-(\omega, t) 
\bigg] dk
\nn\\
&
+
\frac{1}{2\pi} \int_{k\in \mathcal L^-}  
\frac{e^{ik(x+L)}}{1 - e^{4ikL}} \, \bigg[
\frac{\mu_\alpha }{\mu_\delta   \mu_\beta } 
\,  B_{g_0}^-(\omega, t)
+ 
\frac{i\delta k}{1+\delta k^2} \, B_{g_0'}^+(\omega, t)
\nn\\
&\quad
+
\frac{\beta}{\mu_\alpha \mu_\delta  \mu_\beta } \, B_{g_0''}^-(\omega, t)
- \frac{\alpha}{\mu_\alpha \mu_\delta   \mu_\beta } 
\, 
B_{g_2}^-(\omega, t)
-
\frac{\beta \delta}{\mu_\alpha \mu_\delta  \mu_\beta } \, B_{g_2''}^-(\omega, t)
\bigg]
\, dk
\nn\\
&
+
\frac{1}{2\pi} \int_{k\in\mathcal L^+} \frac{e^{ik(x-L)}}{1-e^{-4ikL}}
\, \bigg[
-\frac{\mu_\alpha }{\mu_\delta   \mu_\beta } \, 
B_{h_0}^-(\omega, t)
-
\frac{i\delta k}{1+\delta k^2} \,
B_{h_0'}^+(\omega, t)
\nn\\
&\quad
-
\frac{\beta}{\mu_\alpha \mu_\delta  \mu_\beta } \, 
B_{h_0''}^-(\omega, t)
+
\frac{\alpha}{\mu_\alpha \mu_\delta   \mu_\beta } \, 
B_{h_2}^-(\omega, t)
+
\frac{\beta \delta}{\mu_\alpha \mu_\delta  \mu_\beta }  \, B_{h_2''}^-(\omega, t)
\bigg]
\, dk\ .
\end{align}
\end{subequations}

\begin{remark}[Solution to the linear Nwogu system on a finite interval]\label{non-reg-r}
If $\beta=0$ then $\mu_\beta \equiv 1$, $\omega = k \mu_\alpha/\mu_\delta$ and formulas \eqref{sols-comb} provide the solution to the following initial-boundary value problem for the linearization of the original  Nwogu system posed on a finite interval:
\begin{equation}\label{lnwog-ibvp}
\begin{aligned}
&
\begin{aligned}
&\eta_t + u_x - \alpha u_{xxx}  = 0\ , 
\\  
&u_t + \eta_x - \delta u_{xxt} = 0\ ,
\end{aligned}
\quad (x, t) \in (-L,L) \times (0, T)\ ,
\\
&\eta(x, 0) = \eta_0(x), \  u(x, 0) = u_0(x)\ ,
\\
&u(-L, t) = g_0(t), \ u(L, t) = h_0(t)\ , 
\quad
u_{xx}(-L, t) = g_2(t), \ u_{xx}(L, t) = h_2(t)\ .
\end{aligned}
\end{equation}
Moreover, as described in Remark \ref{adm-bc-r}, following our methodology it is possible to obtain the explicit solution to problem \eqref{lnwog-ibvp} when the boundary conditions are replaced by the second set in \eqref{adm-bc-i}. 
\end{remark}

\section{Well-posedness of the nonlinear problem}
\label{sec:wp-s}

In this section, we prove that the initial-boundary value problem \eqref{eq:Nwogu3nd} for the regularized Nwogu system on a finite interval is well-posed in appropriate function spaces. This problem corresponds to the case of homogeneous, reflective boundary conditions. 
%
%

Taking $\varepsilon = 1$ without loss of generality, the weak formulation of the regularized system in \eqref{eq:Nwogu3nd} is 
\begin{equation}\label{eq:weakd}
\begin{aligned}
&\eta_t+ (I-b\partial_x^2)^{-1}\partial_x\left(u+ \eta u +a u_{xx}\right)=0\ ,\\
&u_t+ (I-d\partial_x^2)^{-1}\partial_x\left(\eta+\tfrac{1}{2}u^2 \right)= 0\ ,
\end{aligned}
\end{equation}
where $\partial_x^j$ denotes the $j$-th generalized derivative with respect to $x$ for $j\geq 1$ where we omit the index in case $j=1$.
In the first equation of \eqref{eq:weakd}, the operator $(I-b\partial_x^2)^{-1}$ is the inverse of $(I-b\partial_x^2)$ with domain  $H^1(-L,L)$, while the operator $(I-d\partial_x^2)^{-1}$ is the inverse of the operator $(I-d\partial_x^2)$ with domain $X=H^2\cap H^1_0(-L,L)$ where $H^1_0(-L,L):=\{v\in H^1(-L,L):~ v(L)=v(-L)=0\}$. 
The operator $L_N:=(I-b \partial_x^2)^{-1}\partial_x$ is realized as the convolution
\begin{equation}\label{eq:green1a}
(L_Nf)(x)=\int_{-L}^{L}F_\xi(x,\xi) f(\xi)~d\xi\ ,
\end{equation}
where $F$ is the Green's function for the two-point boundary value problem
\begin{equation}\label{eq:green1b}
\begin{aligned}
&w-b w_{xx}=-f_x,\quad x\in (-L,L)\ ,\\
&w_x(-L)=w_x(L)=0\ ,
\end{aligned}
\end{equation}
and is defined for $x,\xi\in [-L,L]$ as (e.g., see \cite{zl2013})
$$
F(x,\xi):=-\frac{1}{b W}
\left\{
\begin{array}{lr}
\omega_1(\xi)~\omega_2(x), & -L\leq \xi\leq x,\\ 
\omega_1(x)~\omega_2(\xi), & x<\xi\leq L,
\end{array}\right.
$$
where $\omega_1(x)=\cosh\left(\frac{L+x}{\sqrt{b}}\right)$, $\omega_2(x)=\cosh\left(\frac{L-x}{\sqrt{b}}\right)$, and $W=\omega_1\omega_2'-\omega_1'\omega_2=-\frac{1}{\sqrt{b}}\sinh\left(\frac{2L}{\sqrt{b}}\right)$. 

\begin{remark}\label{rem:remark21}
Note that if $f\in C^1$ with $f(-L)=f(L)=0$, then the classical solution of problem \eqref{eq:green1b} can be written in the form \eqref{eq:green1a} with the help of Green's functions after integration by parts.
\end{remark} 

Similarly, the operator $L_D:=(I-d \partial_x^2)^{-1}\partial_x$ is realized as the convolution
\begin{equation}\label{eq:green2a}
(L_Df)(x)=\int_{-L}^{L}G_\xi(x,\xi) f(\xi)~d\xi\ ,
\end{equation}
where $G$ is the Green's function for the two-point boundary value problem
\begin{equation}\label{eq:green2b}
\begin{aligned}
&w-d w_{xx}=-f_x,\quad x\in [-L,L]\ ,\\
&w(-L)=w(L)=0\ ,
\end{aligned}
\end{equation}
and is defined for $x,\xi\in [-L,L]$ as
$$
G(x,\xi):=-\frac{1}{d W}\left\{\begin{array}{lr}
\omega_1(\xi)~\omega_2(x), & -L\leq \xi\leq x,\\ 
\omega_1(x)~\omega_2(\xi), & x<\xi\leq L,\end{array}\right.
$$
where $\omega_1(x):=\sinh\left(\frac{L+x}{\sqrt{d}}\right)$, $\omega_2(x)=\sinh\left(\frac{L-x}{\sqrt{d}}\right)$, and $W=\omega_1\omega_2'-\omega_1'\omega_2=-\frac{1}{\sqrt{d }}\sinh\left(\frac{2L}{\sqrt{d}}\right)$.
(The subscripts $N$ and $D$ in the notation of $L_N$ and $L_D$ denote Neumann and Dirichlet boundary conditions, respectively.)

The following continuity properties of the operators $L_N$ and $L_D$ have been established in \cite{ADM2009}.
\begin{lemma}\label{lem:continuity}
Let $L_N$ and $L_D$ be the operators defined by \eqref{eq:green1a} and \eqref{eq:green2a}.
\begin{enumerate}[(i)]
    \item If $v\in L^2$, then $L_Nv\in H^1$ and $\|L_Nv\|_1\leq M \|v\|$, where $M>0$ depends on $b$.
    \item If $v\in C^m$, $m\geq 0$, then $L_Nv\in C^{m+1}$ and $\|L_Nv\|_{C^{m+1}}\leq M\|v\|_{C^m}$, where $M>0$ depends on $m, b$.
    \item If $v\in L^2$, then $L_Dv\in H^1_0$ and $\|L_Dv\|_1\leq M\|v\|$, where $M>0$ depends on $d$.
    \item If $v\in H^1$, then $L_Dv\in H^2$ and $\|L_Dv\|_2\leq M\|v\|_1$, where $M>0$ depends on $d$.
    \item If $v\in C^m$, $m\geq 0$, then $L_Dv\in C^{m+1}$ and $\|L_D v\|_{C^{m+1}}\leq M\|v\|_{C^m}$, where $M>0$ depends on $m, d$.
\end{enumerate}
\end{lemma}

Starting from \eqref{eq:green1a} and integrating by parts while using the boundary conditions  $u(\pm L, t) = 0$ and the fact that $\omega_1, \omega_2$ satisfy the homogeneous counterpart of problem \eqref{eq:green1b}, we find that $L_Nu_{xx}=\frac{1}{b}u_x+\frac{1}{b}L_N u$. In turn, we can rewrite system \eqref{eq:weakd} in terms of the operators $L_N$, $L_D$ as
\begin{equation}\label{eq:regweak}
\eta_t=L_N(\tfrac{a+b}{b}u+\eta u)+\tfrac{a}{b}u_x\ ,\quad
u_t=L_D(\eta+\tfrac{1}{2}u^2)\ ,
\end{equation}
which after integration in time leads to the system of integral  equations
\begin{equation}\label{eq:regweak2}
\begin{aligned}
&\eta(x,t)=\eta_0(x)+\int_0^t \left[ L_N(\tfrac{a+b}{b}u+\eta u)+\tfrac{a }{b}u_x \right] d\tau\ ,
\\
&u(x,t)=u_0(x)+\int_0^tL_D(\eta+\tfrac{1}{2}u^2)~d\tau\ .
\end{aligned}
\end{equation}

\begin{remark}
It is implied immediately that any classical solution of the initial-boundary value problem~\eqref{eq:Nwogu3nd} is a weak solution of system \eqref{eq:weakd}.
\end{remark}

First, we establish the uniqueness of solutions of system \eqref{eq:regweak2} in the spaces $H^1_T:=C(0,T;H^1)$ and $H^2_T:=C(0,T;H^2\cap H^1_0)$. These are Banach spaces and for $u\in H^s_T$ their norm is defined as $\|u\|_{H^s_T}=\sup_{0\leq t\leq T}\|u(\cdot,t)\|_s$.

\begin{proposition}[Uniqueness]
\label{prop:uniqueness}
Let $0<T<\infty$ and $(\eta_0,u_0)\in H^1\times (H^2\cap H^1_0)$. Then, system \eqref{eq:regweak2} has at most one solution $(\eta,u)\in  H_T^1 \times H_T^2$.
\end{proposition}

\begin{proof}
Suppose that system \eqref{eq:regweak2}  has two solutions  $(\eta_1,u_1)$ and $(\eta_2,u_2)$   in  $H_T^1 \times H_T^2$. Then, the differences $\eta=\eta_1-\eta_2$ and $u=u_1-u_2$ satisfy
$$
\eta(x, t) = \int_0^t \left[L_N(\tfrac{a+b}{b}u+\eta u_1+\eta_2u)+\tfrac{a}{b}u_x\right]d\tau\ ,
\quad
u(x, t) = \int_0^t L_D\left(\eta+\tfrac{1}{2}u(u_1+u_2)\right) d\tau\ ,
$$
with $\eta(x, 0) = u(x, 0) = 0$. 
Thus, using the (algebra property of $H^s$) inequality $\|fg\|_s\leq \|f\|_s\|g\|_s$ for any $f,g\in H^s$ with $s>1/2$ together with   Lemma \ref{lem:continuity}, we have
$$\|\eta\|_1+\|u\|_2\lesssim \int_0^t\|\eta(\tau)\|_1+\|u(\tau)\|_2~d\tau\ .$$
Then, Gr\"onwall's inequality   implies  $\|\eta\|_1+\|u\|_2\lesssim \|\eta(0)\|_1+\|u(0)\|_2 = 0$ and so the solution, if it exists, is unique. 
\end{proof}

Next, we establish the existence of solutions to \eqref{eq:regweak2}. 
\begin{proposition}[Existence]\label{prop:weaksol1}
Suppose that $(\eta_0,u_0)\in H^1\times (H^2\cap H^1_0)$. Then, there exists time $T>0$ such that system \eqref{eq:regweak2} has a unique solution $(\eta,u)\in  H_T^1\times  H_T^2$.
\end{proposition}
\begin{proof}
Let $E$ denote the Banach space $H^1_T\times H^2_T$ with norm $\|(v,w)\|_E=\|v\|_{H^1_T}+\|w\|_{H^2_T}$. Moreover, for $R:= 2\|(\eta_0,u_0)\|_E$, let $B_R=\{(v,w)\in E: \|(v,w)\|_E\leq R\}$ be the ball of radius $R>0$ centered at the origin of $E$.
Then, the mapping $\Gamma:E\to E$ given by
$$
\Gamma(v,w)=\left(\eta_0+\int_0^t \left[L_N(\tfrac{a+b}{b}w+v w)+\tfrac{a}{b}w_x\right] d\tau, \ u_0+\int_0^tL_D(v+\tfrac{1}{2}w^2)~d\tau\right)
$$
is well-defined, as one may deduce from the regularity of the operators $L_N, L_D$ and of that of the initial conditions $\eta_0, u_0$ and of $v, w$. If $(\eta_1,u_1), (\eta_2,u_2) \in B_R$, then
$$
\begin{aligned}
\left\|\Gamma(\eta_1,u_1)-\Gamma(\eta_2,u_2)\right\|_E 
&\leq 
\left\|\int_0^T \left[ L_N\left\{\tfrac{a+b}{b}(u_1-u_2)+\eta_1(u_1-u_2)+u_2(\eta_1-\eta_2)\right\}+\tfrac{a}{b}(u_1-u_2)_x \right] d\tau\right\|_1
\\
& \quad +
\left\|\int_0^tL_D\left(\eta_1-\eta_2+\tfrac{1}{2}(u_1-u_2)(u_1+u_2)\right)~d\tau\right\|_2
\\
&\leq CT\left[(1+\|\eta_1\|_{H^1_T})\|u_1-u_2\|_{H^2_T}+\|u_2\|_{H^1_T}\|\eta_1-\eta_2\|_{H^1_T}+\|u_1-u_2\|_{H^2_T}+\right. \\
&\quad \left. + \|\eta_1-\eta_2\|_{H^1_T}+\|u_1-u_2\|_{H^2_T}\|u_1+u_2\|_{H^1_T}\right]
\\
& \leq CT\left(2+4R\right) \|(\eta_1,u_1)-(\eta_2,u_2)\|_E\ .
\end{aligned}
$$
Hence, choosing $T=1/[2C(2+4R)]$ implies 
\begin{equation}\label{contr-ineq}
\left\|\Gamma(\eta_1,u_1)-\Gamma(\eta_2,u_2)\right\|_E \leq \frac 12  \|(\eta_1,u_1)-(\eta_2,u_2)\|_E. 
\end{equation}
Moreover, if $(\eta,u)\in B_R$ then by the triangle inequality, the inequality \eqref{contr-ineq} and the definition of $R$,
\begin{equation}\label{into-ineq}
\|\Gamma(\eta,u)\|_E 
\leq \|\Gamma(\eta,u)-\Gamma(0,0)\|_E+\|\Gamma(0,0)\|_E
\leq \frac 12  R+\|(\eta_0,u_0)\|_E = R,
\end{equation}
which shows that $\Gamma$ maps $B_R$ into $B_R$. Together, the inequalities \eqref{contr-ineq} and \eqref{into-ineq} imply that $\Gamma:B_R\to B_R$  is a contraction  and so by  the contraction mapping theorem $\Gamma$ has a unique fixed point in $B_R$. 
\end{proof}

\begin{remark}[Continuous dependence on the data]
As usual in the proof of well-posedness via a contraction mapping argument, the continuous dependence of the data-to-solution map (which is the third component of Hadamard well-posedness) readily follows from the contraction inequality \eqref{contr-ineq}.
\end{remark}

We now establish the existence of solutions of \eqref{eq:regweak2} in spaces of smooth functions. Specifically, for $m\geq 1$ we consider the spaces
$C_0^{m+1}:=\left\{w\in C^{m+1}: w( \pm L)=0\text{ and } w_{xx}( \pm L)=0\right\}$
and 
$\widetilde{C}_0^m=\{v\in C^m: v_x(\pm L)=0\}$.

Then, we have the following result.
\begin{proposition}[Smooth solution]\label{prop:smoothsol}

Given $m\geq 1$ and initial conditions $(\eta_0,u_0)\in \widetilde{C}^m_0\times C^{m+1}_0$, there exists time $T_m>0$ such that the system \eqref{eq:regweak2} has a unique solution $(\eta,u)\in C(0,T_m;\widetilde{C}^m_0)\times C(0,T_m;C^{m+1}_0)$.
\end{proposition}
\begin{proof}

If $(\eta_0,u_0)\in \tilde{C}^m_0\times C^{m+1}_0$ with $m\geq 1$, then using the same arguments as in the proof of Proposition~\ref{prop:uniqueness}  we can show uniqueness of solution to \eqref{eq:regweak2} in the space $C(0,T;\tilde{C}^m_0)\times C(0,T;C^{m+1}_0)$.

Consider the Banach space $E_m:=C(0,T_m;\widetilde{C}_0^m)\times C(0,T_m;C_0^{m+1})$. Following similar arguments as in Proposition \ref{prop:weaksol1}, define the mapping $\Gamma_m$ for $(v,w)\in E_m$ as
$$\Gamma_m(v,w)=\left(\eta_0+\int_0^t \left[ L_N(\tfrac{a+b}{b}w+v w)+\tfrac{a}{b}w_x\right] d\tau, \ u_0+\int_0^tL_D(v+\tfrac{1}{2}w^2)~d\tau\right)\ .$$
Define also $$\phi(x,t)=\eta_0(x)+\int_0^t \left[L_N(\tfrac{a+b}{b}w+v w)+\tfrac{a}{b}w_x\right]d\tau\ .$$ Observe that $\phi_x(\pm L,t)=0$ because  $\eta_0 \in \widetilde{C}_0^m$, $w \in C_0^{m+1}$ and $L_N f$ satisfies the boundary value problem~\eqref{eq:green1b}. Similarly, define
$$
\psi(x,t)=u_0+\int_0^tL_D(v+\tfrac{1}{2}w^2)~d\tau\ ,
$$
and note that $\psi(\pm L,t)=0$ because $u_0 \in C_0^{m+1}$ and $L_D f$ satisfies the boundary value problem~\eqref{eq:green2b}. Furthermore, in our case where $f=v+\tfrac{1}{2}w^2$ and $v_x(\pm L, t)=w(\pm L, t)=0$, we have that $f_x(\pm L, t)=v_x(\pm L, t)+w(\pm L, t)w_x(\pm L, t)=0$ and so
$$(L_Df)_{xx}(\pm L, t)=\tfrac{1}{d}\left[(L_Df)(\pm L, t)+f_x(\pm L, t)\right]=0\ .$$ Hence, the second set of boundary conditions $\psi_{xx}(\pm L,t)=0$ also hold true. Thus, by Lemma \ref{lem:continuity} we deduce that the mapping $\Gamma_m:E_m\to E_m$ is well-defined. The rest of the proof follows as in  Proposition~\ref{prop:weaksol1}.
\end{proof}

Finally, we show that the smooth solution of system \eqref{eq:regweak2} guaranteed by Proposition \ref{prop:smoothsol} is the classical solution of system \eqref{eq:Nwogu3nd}. 
\begin{proposition}\label{prop:main1}
For $(\eta_0,u_0)\in \widetilde{C}^2_0\times C^3_0$, let $(\eta,u)\in H^1_T\times H^2_T$ be the solution of   system \eqref{eq:regweak2} as guaranteed by Proposition \ref{prop:smoothsol}. Then, $(\eta,u)\in C(0,T;\widetilde{C}^2_0)\times C(0,T;C^3_0)$ satisfies system \eqref{eq:Nwogu3nd} pointwise in $[0,T]$ and $(\eta_t,u_t)\in C(0,T;\widetilde{C}^2_0)\times C(0,T;C^3_0)$.
\end{proposition}
\begin{proof}
By Proposition \ref{prop:smoothsol}, there exists time $T_2>0$ that depends only on the initial condition and the parameters $a, b, d$ (actually only on $a$ and $d$ as one may obtain upper bounds of the solutions independent of $b$)  such that a unique solution of  system \eqref{eq:regweak2} exists in $E_2$. From system \eqref{eq:regweak2} we observe that $\eta$ and $u$ are differentiable with respect to time $t$ and that $\eta_t$ and $u_t$ are given by  \eqref{eq:regweak}. In addition, from \eqref{eq:regweak2} we see that $\|\eta\|_{C^2}+\|u\|_{C^3}$ is always bounded and thus the solution $(\eta,u)$ of Proposition \ref{prop:smoothsol} belongs to $C(0,T;\widetilde{C}^2_0)\times C(0,T;C^3_0)$. Moreover, system \eqref{eq:regweak} and Lemma \ref{lem:continuity} imply $(\eta_t,u_t)\in C(0,T;\widetilde{C}^2_0)\times C(0,T;C^3_0)$. Observing now that, in view of \eqref{eq:weakd}, 
$$
\eta_t-b \eta_{xxt}=L_N(u+\eta u +a  u_{xx})-b L_N(u+\eta u +a u_{xx})_{xx}= -\left(u+\eta u +a u_{xx}\right)_x$$
and, similarly, 
$u_t-bu_{xxt}=-(\eta+\frac{1}{2}u^2)_x$, 
we conclude that a weak solution of \eqref{eq:weakd} is also the classical solution of \eqref{eq:Nwogu3nd} for appropriate initial conditions.
\end{proof}

We conclude this section with an estimate for the maximum time of existence for the regularized Nwogu system  with wall-boundary conditions \eqref{eq:Nwogu3nd} where here we take into account the influence of the parameter $\varepsilon$.  
Like any other representative of the $abcd$-systems \eqref{eq:coeffs}, in the case of wall-boundary conditions the regularized Nwogu system  satisfies the mass conservation law
$$
\frac{d}{dt}M(t;v)=0, \quad  M(t;v):=\int_{-L}^L v(x,t)~ dx\ ,
$$
for both $v=\eta$ and $v=u$. The solution of the initial-boundary value problem \eqref{eq:Nwogu3nd} in the special case where $b=d>0$    preserves additionally the energy functional 
\begin{equation}\label{energy}
E(t)=\frac{1}{2}\int_{-L}^L \left[\eta^2+(1+\varepsilon\eta)u^2-\varepsilon au_x^2\right] dx\ ,
\end{equation}
in the sense that $E(t)=E(0)$ for $t\geq 0$. This can be seen by writing system \eqref{eq:Nwogu3nd} in the form
$$
\eta_t+P_x=0, \quad u_t+Q_x=0\ ,
$$
where $P=u+\varepsilon \eta u+\varepsilon a u_{xx}-\varepsilon b\eta_{xt}$ and $Q=\eta+\varepsilon\frac{1}{2}u^2-\varepsilon du_{xt}$. Then, note that $P(-L)=P(L)=0$ for all $t\geq 0$. Thus, multiplying the first equation by $Q$ and the second one by $P$, and  then adding the resulting equations, we obtain  via integration by parts the desired functional \eqref{energy}.

Systems that satisfy the previous energy conservation function \eqref{energy} have 
$\nu=\frac{2(3\theta^2-2)}{3\theta^2-1}$ and $\mu=0$, 
with the parameters $a,b,d$  given by  \eqref{eq:coeffs} as \begin{equation}\label{eq:coeffs2}
a=\theta^2-\tfrac{2}{3},\quad b=d=\tfrac{1}{2}(1-\theta^2)\ , \quad \theta^2<\tfrac{2}{3}.
\end{equation}
The energy functional \eqref{energy} is of order three, and thus does not give directly any useful energy estimates. On the other hand, if we multiply the first equation of \eqref{eq:Nwogu3nd} by $\eta$ and the second one by $u+a u_{xx}$, and then integrate and add the resulting equations, we obtain 
\begin{equation}\label{eq:enerest}
\frac{1}{2}\frac{d}{dt}\int_{-L}^L[\eta^2+\varepsilon b \eta_x^2+u^2+\varepsilon(d-a)u_x^2-\varepsilon^2 adu_{xx}^2]~d  x =\frac{1}{2}\int_{-L}^{L}[-2\varepsilon^2 auu_xu_{xx}-\varepsilon u_x\eta^2]~d x\ .
\end{equation}
Following \cite{ADM2009}, we have  
\begin{align*}
&\varepsilon \left| \int_{-L}^L \eta^2 u_x~dx\right|\leq C\varepsilon \|\eta\|\|\eta_x\|\|u_x\|\leq C\left(\|\eta\|^2+\varepsilon\|\eta_x\|^2+\varepsilon\|u_x\|^2\right)^{\frac 32}\ ,
\\
&\varepsilon^2 \left|\int_{-L}^Lu u_x u_{xx}~dx\right| \leq C\varepsilon^2\|u_x\|^2\|u_{xx}\|\leq C\left(\varepsilon \|u_x\|^2+\varepsilon \|u_{xx}\|^2 \right)^{\frac 32}\ .
\end{align*}
Then, letting
$$I_\varepsilon=I_\varepsilon(t):= \int_{-L}^L[\eta^2+\varepsilon b \eta_x^2+u^2+\varepsilon(d-a)u_x^2-\varepsilon^2 adu_{xx}^2]~d  x\ ,$$
equality \eqref{eq:enerest} leads to $\dfrac{d}{dt}I_\varepsilon\leq CI_\varepsilon^{3/2}$, %
which is a differential inequality of Bernoulli type with solution
$$I_\varepsilon(t)\leq \frac{I_\varepsilon(0)}{\big(1-Ct\sqrt{I_\varepsilon(0)}\,\big)^2}\leq \frac{I_1(0)}{\big(1-Ct\sqrt{I_0(0)}\,\big)^2}\ ,$$
for sufficiently small $t$ and $\varepsilon\in [0,1]$. This provides an {\em a priori} $H^1\times H^2$ bound, independent of $\varepsilon$ (and $d$), for the solution. From this, we conclude that the maximal existence time $T$ will be \textit{independent} of $\varepsilon$. It is worth mentioning that there is \textit{no} global well-posedness result known for Nwogu-type systems, even for the Cauchy problem~\cite{BCS2004}.

\begin{remark}
One may extend in a straightforward manner the Proposition \ref{prop:main1} for $(\eta_0,u_0)\in\widetilde{C}^m_0\times C^{m+1}_0$, $m\geq 2$ and conclude that system \eqref{eq:Nwogu3nd} with $a<0$ and $b,d>0$ has a unique solution local in time with $(\partial^k_t\eta,\partial^k_t u)\in \widetilde{C}^m_0\times C^{m+1}_0$ for all $k\geq 0$, $m\geq 2$. 
\end{remark}


\section{A modified Galerkin method for the numerical discretization of the regularized Nwogu system}
\label{num-s}

In this section, we consider a modified Galerkin method for the initial-boundary value problem~\eqref{eq:Nwogu3nd}. For simplicity, we take again $\varepsilon=1$. Let $-L=x_0<x_1<\ldots<x_N=L$ be a uniform grid of the domain $[-L,L]$ with grid size $h=\Delta x=x_{i+1}-x_i=2L/N$. For integers $r\geq 2$ and $0\leq \mu\leq r-2$ we define the finite element spaces 
$$
S_h=S_h(\mu,r)=\left\{\phi\in C^\mu[-L,L]:\phi\big|_{[x_i,x_{i+1}]}\in\mathbb{P}_{r-1}\right\}\ ,
\quad
S_h^0=\left\{\phi\in S_h:\phi(\pm L)=0\right\}\ ,
$$
where $\mathbb{P}_q$ denotes the space of polynomials of degree at most $q$. The space $S_h$ is a subspace of $H^{\mu+1}$ while the space $S_h^0$ is subspace of $H^{\mu+1}\cap H^1_0$. For example, the space of cubic splines will be the space $S_h(2,4)$, and the space of Lagrange polynomials of order $q=r-1$ with $r\geq 2$ will be the space $S_h(0,r)$. In general, it is well known (see for example \cite{BF1973,S1980}) that  $S_h$ and $S_h^0$ have the following approximation properties.
\begin{lemma}
Let $r\geq 2$, $0\leq \mu\leq r-2$, and $m$, $k$ be integers such that $0\leq m\leq \mu+1$ and $m<k\leq r$. Then, there exists a constant $C$, independent of $h$, such that
\begin{align*}
&\min_{\chi\in S_h}\|(w-\chi)^{(m)}\|\leq Ch^{k-m}\|w^{(k)}\|\quad \text{for}\quad w\in H^k\ ,
\\
&\min_{\chi\in S_h}\|(w-\chi)^{(m)}\|_\infty \leq Ch^{k-m}\|w^{(k)}\|_\infty \quad \text{for}\quad  w\in C^k\ .
\end{align*}
Both results are valid also in the case of $S_h^0$ when $w\in H^k\cap H_0^1$ and $w\in C_0^k$, respectively.
\end{lemma}

In what follows, we describe the modified Galerkin semidiscretization of system \eqref{eq:Nwogu3nd} inspired by \cite{WB1999}. Similar procedures have been used for the discretization of other Boussinesq-type equations including the Serre and Bona-Smith systems  \cite{MSM2017,DMS2010}.

\subsection{Semidiscretization}\label{sec:estimates}

First, consider the bilinear forms $A_q:H^1\times H^1\to\mathbb{R}$ such that for all $u,v\in H^1$, $A_q(u,v)=(u,v)+ q(u_x,v_x)$ for $q=b>0$ and $q=d>0$.
These particular bilinear forms are symmetric, which will result in symmetric matrices in the full-discretization. They are also bounded and coercive. 

Assume that the maximal time of existence of solutions to  problem \eqref{eq:Nwogu3nd} guaranteed by Proposition~\ref{prop:smoothsol}  is $T>0$. 
We define the discrete Laplacian to be the operator $\partial^2_h:H^1_0\to S_h^0$ such that
$$
(\partial_h^2 \tilde{u},\phi)=-(\tilde{u}_x,\phi_x)\quad \forall \phi\in S_h^0\ .
$$
The modified Galerkin method is defined as follows. We seek approximations $\tilde{\eta}:[0,T]\to S_h$ and $\tilde{u}:[0,T]\to S_h^0$ of $\eta$ and $u$, respectively, such that
\begin{equation}\label{eq:semidiscrete}
\begin{aligned}
A_b(\tilde{\eta}_t,\chi)+([\tilde{u}+\tilde{\eta}\tilde{u}+a \partial^2_h \tilde{u}]_x,\chi)=0 \quad \forall\chi\in S_h\ ,\\
A_d(\tilde{u}_t,\phi)+(\tilde{\eta}_x+\tilde{u}\tilde{u}_x,\phi)= 0 \quad \forall\phi\in S_h^0\ ,
\end{aligned}\quad 0<t\leq T\ ,
\end{equation}
with initial conditions
$\tilde{\eta}(x,0)=R_h\eta_0(x)$,  $\tilde{u}(x,0)=R_h^0u_0(x)$, 
where $R_h^0:H^1_0\to S_h^0$ and $R_h:H^1\to S_h$ are the elliptic projections such that for any $v\in H^1_0$ and $w\in H^1$ they satisfy $A_d(R_h^0v,\phi)=A_d(v,\phi)$ for all $\phi\in S_h^0$
and $A_b(R_hw,\chi)=A_b(w,\chi)$ for all $\chi\in S_h$, respectively.
It is known that the elliptic projections  $R_h^0$ and $R_h$ have favorable stability and convergence properties in addition to its optimal accuracy (see \cite{ADM2010} and references therein).

The results we will use are summarized in the following lemma.
\begin{lemma}\label{lem:ellipt}
If $\mu\geq 0$ and $k=0$ or $1$, then
\begin{enumerate}[(i)]
\item the projection $R_h^0$ is stable in $L^2$ and $H^1$. In particular,  there exists a constant $C$ independent of $h$ such that $\|R_h^0v\|_k \leq C\|v\|_k \quad \forall v\in H_0^1$, and
 \item  it satisfies the the optimal error estimate $\|R_h^0v-v\|_k\leq Ch^{r-k}\|v\|_r$ for all  $v\in H^r\cap H^1_0$.
\end{enumerate}
Similar results hold true for the elliptic projection $R_h$.
\end{lemma}

The standard $L^2$-projection $P_h:L^2 \to S_h$ is defined as $(P_h w,\chi)=(w,\chi)$ for all $\chi\in S_h$ and $w\in L^2$. Similarly, the $L^2$-projection $P_h^0:L^2 \to S_h^0$ is defined as $(P_h^0 u,\phi)=(u,\phi)$ for all $\phi\in S_h^0$ and $u\in L^2$. It is known that the $L^2$-projection satisfies the error estimate $\|P_hv-v\|_\infty\leq Ch^r\|v\|_{r,\infty}$ for all $v\in H^r$. This result is also true for the $P_h^0$ projection.

Define the mappings $f_h:L^2\to S_h$ and $g_h:L^2\to S_h^0$ such that 
\begin{equation*}
A_b(f_h[w],\chi)=(w,\chi_x)\quad \forall \chi\in S_h\ ,
\quad \text{and}\quad
A_d(g_h[w],\phi)=(w,\phi_x)\quad \forall \phi\in S_h^0\ 
\end{equation*}
and also let 
$$
\|w\|_{-1}:=\sup_{{z\in H^1}\atop{z\not=0}}\frac{(w,z)}{\|z\|_1},
\quad
\|w\|_{-2}:=\sup_{{z\in H^2\cap H^1_0}\atop{ z\not=0}}\frac{(w,z)}{\|z\|_2}.
$$
Then, we have the following result.
\begin{lemma}
Both functionals $f_h$ and $g_h$ satisfy the following stability properties:
\begin{enumerate}[(i)]
    \item $\|f_h[w]\|_1\lesssim \|w\|$ and $\|g_h[w]\|_1\lesssim \|w\|$ for all $w\in L^2$,
    \item  $\|f_h[\psi]\|\lesssim \|\psi\|_{-1}$  for all $\psi\in S_h^0$.
\end{enumerate}
\end{lemma}
\begin{proof}
The proof of (i) follows immediately from the definition of $f_h$ and $g_h$ and the properties of the bilinear forms $A_d$ and $A_b$. 

To prove (ii),  consider the operator $A=I-b\partial_x^2$ with domain $X=\left\{w\in H^2: w_x(\pm L)=0\right\}\subset H^2$, and the problem $Aw=\psi_x$ with $w_x(\pm L)=0$ and $\psi\in S_h^0$. Then,
\begin{equation}\label{eq:boundw}
\|\psi_x\|_{-2} 
= \sup_{{z\in H^2}\atop{z\not=0}} \frac{(\psi_x,z)}{\|z\|_2}=\sup_{{z\in H^2}\atop{z\not=0}} \frac{(Aw,z)}{\|z\|_2} 
\geq \frac{(Aw,A^{-1}w)}{\|A^{-1}w\|_2}
=\frac{\|w\|^2}{\|A^{-1}w\|_2}\geq C\frac{\|w\|^2}{\|w\|} =C\|w\|\ .
\end{equation}
Moreover, for any $z\in H^2$, $z\not=0$ (and because $\psi\in S_h^0$) we have 
$$\frac{(\psi_x,z)}{\|z\|_2}=\frac{-(\psi ,z_x)}{\|z\|_2}\leq \frac{\|\psi \|_{-1}\|z_x\|_1}{\|z\|_2}\leq \|\psi \|_{-1}\ .$$
Hence, 
$\|\psi _x\|_{-2}=\sup_{{z\in H^2}\atop{z\not=0}} \frac{(\psi _x,z)}{\|z\|_2}\leq \|\psi \|_{-1}\ ,$
and by \eqref{eq:boundw} we have $\|w\|\lesssim \|\psi \|_{-1}$.
Also, for $\chi\in S_h$ we have
$A_b(R_hw,\chi)=A_b(w,\chi)=(\psi_x,\chi)=-(\psi,\chi_x)=-A_b(f_h[\psi],\chi)$
and thus $R_hw =-f_h[\psi]$. This implies $\|f_h[\psi]\|=\|R_h w\|\lesssim \|w\|\lesssim \|\psi\|_{-1}$,  completing the proof.
\end{proof}

The following identity for the discrete Laplacian operator will turn out to be very useful.
\begin{lemma}
For all $w\in H_0^1$ and $d>0$, we have
$\partial_h^2 R_h^0w-P_h^0[w_{xx}]=\tfrac{1}{d}(R_h^0w-P_h^0w)$.
\end{lemma}
\begin{proof}
For any $\phi\in S_h^0$,  we have
$$
\begin{aligned}
(\partial_h^2 R_h^0w,\phi)& =-((R_h^0 w)_x,\phi_x)-\tfrac{1}{d}(R_h^0w,\phi)+\tfrac{1}{d}(R_h^0w,\phi)
= -\tfrac{1}{d}[A_d(R_h^0w,\phi)-(R_h^0w,\phi)]\\
&=-\tfrac{1}{d}[A_d(w,\phi)-(R_h^0w,\phi)]=-\tfrac{1}{d}[(w,\phi)+d(w_x,\phi_x)-(R_h^0w,\phi)]\\
&=\tfrac{1}{d}(R_h^0w-w,\phi)+(w_{xx},\phi)\ ,
\end{aligned}
$$
which implies the desired formula.
\end{proof}

Given the inverse inequality \cite{Thomee}
\begin{equation}\label{eq:inversineq}
\|\chi\|_1\leq C_0 h^{-1}\|\chi\| \quad \forall~ \chi\in S_h\ ,
\end{equation}
where $C_0$ is independent of $h$, we have the following inverse estimate of the discrete Laplacian $\partial_h^2$ operator (see also \cite{Thomee}).
\begin{lemma}\label{lem:stability}
For any $w\in H^1_0$ and $h>0$, we have
$\|\partial_h^2w\|_1\lesssim h^{-2} \|w\|_1$. In addition, for  any $\phi\in S_h$ we have
$\|\partial_h^2 \phi\|\lesssim h^{-2} \|\phi\|$. Finally, for any  $w\in H^2$  we have
$\|\partial_h^2 w\|\lesssim \|w\|_2$.
\end{lemma}
\begin{proof}
By the definition of $\partial_h^2 w$, we have
$
\|\partial_h^2w\|^2 = (\partial_h^2 w,\partial_h^2 w)=-(w_x,(\partial_h^2w)_x)\leq \|w_x\|\|\partial_h^2w\|_1\ .
$
Thus,  $\|\partial_h^2 w\|^2_1\leq C_0^2 h^{-2} \|\partial_h^2w\|^2\leq C_0^2h^{-2}\|w\|_1\|\partial_h^2w\|_1$,
which implies $\|\partial_h^2w\|_1\leq C_0 h^{-2}\|w\|_1$. The second inequality can be proved similarly \cite{Thomee}.
\end{proof}

Furthermore, we have the following lemma.
\begin{lemma}\label{lem:stability2}
For any $u\in H^1_0$ and $h>0$, 
$\|\partial_h^2u\|_{-1}\lesssim \|u_x\|$.
\end{lemma}
\begin{proof}
Let $0\not=w\in H^1$. Then, we have  
$$\frac{(\partial_h^2 u,w)}{\|w\|_1}=\frac{(\partial_h^2 u,P_h^0w)}{\|w\|_1}=\frac{-(u_x,(P_h^0w)_x)}{\|w\|_1}\leq \frac{\|u_x\|~\|P_h^0w\|_1}{\|w\|_1}\leq C\|u_x\|\ ,$$
where $P_h^0$ denotes the $L^2$-projection into $S_h^0$. In the last inequality, we used the stability of the $L^2$-projection in $H^1$, see \cite{CT1987}. Taking the supremum over all $w\in H^1$, $w\not=0$, leads to the desired inequality.
\end{proof}

Also, we can show the following.
\begin{lemma}\label{lem:reguldiscl}
For all $\psi\in S_h^0$ we have $\|\partial_h^2 g_h(\psi)\|\lesssim \|\psi_x\|$.
\end{lemma}
\begin{proof}
Let $\chi\in S_h^0$ and consider the problem
$$
\begin{aligned}
&w-|a|w_{xx}=-\psi_x, \quad x\in (-L,L)\ ,\\
&w(-L)=w(L)=0\ .
\end{aligned}
$$
From the theory of elliptic equations, we have $\|w\|_2\leq \|\psi_x\|$. Moreover,  
$$
A_{|a|}(R_h^0w,\chi) =A_{|a|}(w,\chi) =-(\psi_x,\chi) =(\psi,\chi_x) = A_{|a|}(g_h(\psi),\chi)\ ,
$$
which implies $R_h^0w=g_h(\psi)$. Note that
$$\begin{aligned}
\|\partial_h^2 g_h(\psi)\|^2&=\|\partial_h^2 R_h^0w\|^2=(\partial_h^2 R_h^0 w,\partial_h^2 R_h^0 w)=-((R_h^0 w)_x, (\partial_h^2 R_h^0 w)_x)\\
&=-\tfrac{1}{|a|}\left[(R_h^0w,\partial_h^2 R_h^0 w)+|a|((R_h^0 w)_x, (\partial_h^2 R_h^0 w)_x)-(R_h^0w,\partial_h^2 R_h^0 w)\right]\\
&=-\tfrac{1}{|a|}\left[ A_{|a|}(R_h^0w,\partial_h^2 R_h^0w)-(R_h^0w,\partial_h^2 R_h^0 w)\right]
=-\tfrac{1}{|a|}\left[A_{|a|}(w,\partial_h^2R_h^0w)-(R_h^0w,\partial_h^2 R_h^0 w)\right]\\
&= -\tfrac{1}{|a|}\left[(w,\partial_h^2R_h^0w)-|a|(w_{xx},\partial_h^2R_h^0w)-(R_h^0w,\partial_h^2R_h^0w)\right]\\
&\lesssim (\|w\|+\|w_{xx}\|+\|R_h^0w\|)\|\partial_h^2R_h^0w\|\lesssim \|\psi_x\|\|\partial_h^2R_h^0w\|\ .
\end{aligned}
$$
Therefore, $\|\partial_h^2g_h(\psi)\|\lesssim \|\psi_x\|$ and the proof is complete.
\end{proof}

Note that system \eqref{eq:semidiscrete} can be written in the form
\begin{equation}\label{eq:discode}
\tilde{\eta}_t=f_h[\tilde{u}+ \tilde{\eta}\tilde{u}+a \partial_h^2\tilde{u}]\ ,
\quad
\tilde{u}_t=g_h[\tilde{\eta}+\tfrac{1}{2}\tilde{u}^2]\ ,
\end{equation}
or, in a more compact form, $H_t=F(H)$, where $H=(\tilde{\eta},\tilde{u})^T$ and $F(H)=(f_h[\tilde{u}+ \tilde{\eta}\tilde{u}+a \partial_h^2\tilde{u}],g_h[\tilde{\eta}+\tfrac{1}{2}\tilde{u}^2])^T$. Using Lemma \ref{lem:stability} and the inverse inequality \eqref{eq:inversineq}, it is easy to check that for fixed $h>0$ the function $F:S_h\times S_h^0\to S_h\times S_h^0$ is Lipschitz continuous in the norm $\sqrt{\|\chi\|^2+\|\phi\|_1^2}$ for $(\chi,\phi)\in S_h\times S_h^0$. Thus, there is a maximal time $t_h>0$ such that the system \eqref{eq:discode} with initial conditions $(\tilde{\eta}(0),\tilde{u}(0))=(R_h\eta_0,R_h^0 u_0)$ has a unique solution $(\tilde{\eta},\tilde{u})\in S_h\times S_h^0$ for all $t\in [0,t_h]$.

Similarly, for the solution of   problem \eqref{eq:Nwogu3nd} we have  that for any $\chi\in S_h$
$$
(\eta_t,\chi)+b(\eta_{xt},\chi_x)= (u+\eta u+a u_{xx},\chi_x)=A_b(f_h[u+ \eta u+a u_{xx}],\chi)\ ,
$$
which implies
$A_b(\eta_t,\chi)=A_b(f_h[u+ \eta u+a u_{xx}],\chi)$.
Along the same lines, for all $\phi\in S_h^0$ we have
$$
A_d(u_t,\phi) = (\eta+\tfrac{1}{2}u^2,\phi_x)= A_d(g_h[\eta+\tfrac{1}{2}u^2],\phi)\ .
$$
Therefore, we obtain
\begin{equation}\label{eq:contdisc}
R_h\eta_t=f_h[u+ \eta u+a u_{xx}]\ ,
\quad R_h^0 u_t=g_h[\eta+\tfrac{1}{2}u^2]\ .
\end{equation}

We are now ready to prove the main result of this section.
\begin{proposition}\label{prop:errorest}
For $r\geq 2$, $0\leq \mu\leq r-2$ and $t\in[0,T]$, there is a constant $C$ independent of $h$ such that the solution $(\tilde{\eta},\tilde{u})\in S_h(\mu,r)\times S_h^0(\mu,r)$ of the semidiscrete problem \eqref{eq:semidiscrete} converges to the exact solution of system  \eqref{eq:Nwogu3nd} as $h\to 0$, and satisfies the a priori error estimates 
$\|\tilde{\eta}-\eta\|+\|\tilde{u}-u\|\leq Ch^r$
and 
$\|\tilde{\eta}-\eta\|_{1}+\|\tilde{u}-u\|_{1}\leq Ch^{r-1}$.
\end{proposition}
\begin{proof}
Consider the quantities
$\theta=\tilde{\eta}-R_h\eta$, $\rho=R_h\eta-\eta$, $\zeta=\tilde{u}-R_h^0u$, $\xi = R_h^0u-u$, 
$e_\eta=\tilde{\eta}-\eta=\theta+\rho$, $e_u=\tilde{u}-u=\zeta+\xi$.
Combining systems \eqref{eq:discode} and \eqref{eq:contdisc}, we have

\begin{equation}\label{eq:sys1}
\begin{aligned}
\theta_t &= f_h[(\tilde{u}-u)+(\tilde{\eta}\tilde{u}-\eta u)+a(\partial_h^2 \tilde{u}-u_{xx})]\ ,
\\
\zeta_t &= g_h[\tilde{\eta}-\eta+\tfrac{1}{2}(\tilde{u}^2-u^2)]\ .
\end{aligned}
\end{equation}
Noting that
\begin{align*}
&\tilde{\eta}\tilde{u}-\eta u=\tilde{\eta}(\zeta+\xi)+u(\theta+\rho)=-[\eta(\zeta+\xi)+u(\theta+\rho)+(\theta+\rho)(\zeta+\xi)]\ ,
\\
&\tilde{u}^2-u^2=\tilde{u}(\zeta+\xi)+u(\zeta+\xi)=(\zeta+\xi)^2+2u(\zeta+\xi)\ , 
\end{align*}
we rewrite system \eqref{eq:sys1} as
$$
\begin{aligned}
\theta_t &= f_h[\zeta+\xi-[\eta(\zeta+\xi)+u(\theta+\rho)+(\theta+\rho)(\zeta+\xi)]+a(\partial_h^2 \tilde{u}-u_{xx})] \ ,\\
\zeta_t &= g_h[\theta+\rho+\tilde{u}(\zeta+\xi)+u(\zeta+\xi)]\ .
\end{aligned}
$$
Then, employing Lemmata \ref{lem:ellipt}, \ref{lem:reguldiscl} we find
$$
\begin{aligned}
\|\theta_t\| &\leq \|f_h[\zeta+\xi]\|+\|f_h[\eta(\zeta+\xi)]\|+\|f_h[u(\theta+\rho)]\|+\|f_h[(\theta+\rho)(\zeta+\xi)]\|
+\|f_h[a(\partial_h^2\tilde{u}-u_{xx})]\|\\
&\lesssim \|\zeta+\xi\|+\|\eta(\zeta+\xi)\|+\|u(\theta+\rho)\|+\|(\theta+\rho)(\zeta+\xi)\|
+\|f_h[a(\partial_h^2\tilde{u}-\partial_h^2R_h^0u+\partial_h^2R_h^0u-u_{xx})\|\\
&\lesssim \|\zeta+\xi\|+\|\eta(\zeta+\xi)\|+\|u(\theta+\rho)\|+\|(\theta+\rho)(\zeta+\xi)\|
+\|\partial_h^2(\tilde{u}-R_h^0u)\|_{-1}
+\|\partial_h^2R_h^0u-u_{xx}\|\\
& \lesssim \|\zeta\|+\|\xi\|+\|\theta\|+\|\rho\|+\|\tilde{u}-R_h^0u\|_1+\|\partial_h^2R_h^0u-u_{xx}\|\\
& \lesssim \|\zeta\|+\|\xi\|+\|\theta\|+\|\rho\|+\|\zeta\|_1+\|P_h[u_{xx}]-u_{xx}\|+\|R_h^0u-P_hu\|\\
&\lesssim h^r+\|\theta\|+\|\zeta\|_1\ .
\end{aligned}
$$
Working similarly for $\zeta_t$ but using the $H^1$-norm, we obtain
$$
\begin{aligned}
\|\zeta_t\|_1 &\leq \|g_h[\theta+\rho]\|_1+\|g_h[\tilde{u}(\zeta+\xi)]\|+\|g_h[u(\zeta+\xi)]\|\\
&\lesssim \|\theta\|+\|\rho\|+\|\zeta\|+\|\xi\|\lesssim h^r +\|\theta\|+\|\zeta\|\ .
\end{aligned}
$$
The  two inequalities for $\|\theta_t\|$ and $\|\zeta_t\|_1$ combine to imply
$$
\frac{1}{2}\frac{d}{dt}(\|\theta\|^2+\|\zeta\|^2_1)\leq \|\theta\|\|\theta_t\|+\|\zeta\|_1\|\zeta_t\|_1 \lesssim h^{2r}+\|\theta\|^2+\|\zeta\|_1^2
$$
from which we obtain via  Gr\"onwall's inequality that $\|\theta\|+\|\zeta\|_1\lesssim h^r$ for $0\leq t\leq t_h$ and, subsequently, 
$$\|\tilde{\eta}-\eta\|+\|\tilde{u}-u\|\lesssim h^r\ .$$
Furthermore, for $t\leq t_h$ we have  
$$
\|\tilde{\eta}-\eta\|_{L^\infty}\leq \|\tilde{\eta}+R_h\eta\|_{L^\infty}+\|R_h\eta-\eta\|_{L^{\infty}}\lesssim h^{-1/2}\|\tilde{\eta}-R_h\eta\|+ h^r \lesssim h^{r-1/2},
$$
where we used the inverse inequality $\|\tilde{w}\|_{L^\infty}\leq Ch^{-1/2}\|\tilde{w}\|$ for all $\tilde{w}\in S_h$. Thus, for sufficiently small $h$, 
$\|\tilde{\eta}\|_{L^\infty}\leq \|\eta\|_{L^\infty}+\|\tilde{\eta}-\eta\|_{L^\infty}\leq Ch^{r-1/2}+M<2M$. Similar estimates  hold  true for the variable $\tilde{u}$, and so the maximal time of existence can be extended to $t_h=T$ through contradiction. 
\end{proof}

\subsection{Experimental validation}

We shall now verify the results of Proposition \ref{prop:errorest} by numerical means. Specifically, we consider the method of manufactured solutions for the initial-boundary value problem~\eqref{eq:Nwogu3nd} in the interval $(0,1)$ with $a=-1$, $b=d=1$, and exact solution 
$$\eta(x,t)=e^{2t}\cos(\pi x),\quad u(x,t)=e^t x^2 (x-1)^2\sin(\pi x)\ .$$
We tested Lagrange spaces $S_h(\mu,r)$ with $\mu=0$ and $r\geq 2,3$ and $4$, and also cubic splines with $\mu=2$ and $r=4$. For computation of the relative errors we used the standard $L^2$ and $H^1$ norms, where the integrals were approximated using Gauss-Legendre quadrature with 5 nodes. 
In that respect, for a stepsize $h=\Delta x$ we define the normalized numerical errors for the $L^2$ and $H^1$ norms as
$$
E_q(\eta; h)=\frac{\|\tilde{\eta}-\eta\|_q}{\|\eta\|_q}, \quad E_q(u; h)=\frac{\|\tilde{u}-u\|_q}{\|u\|_q}\ , \quad q=0,1\ .
$$
For two different stepsizes $h_1$ and $h_2$, we define the experimental convergence rate as the ratio
$$
\tilde{r}_q=\frac{\log(E_q(v;h_1))/\log(E_q(v;h_2))}{\log(h_1)/\log(h_2)}\ .
$$
The experimental convergence rate $\tilde{r}_0$ is expected to converge to the corresponding value of $r$. For the integration in time, we used the four-stage classical Runge-Kutta method of order four with very small stepsize $\Delta t$ in order to ensure that the corresponding errors are negligible compared to the errors due to the spatial discretization. All the computer codes were written in Fortran with double precision arithmetic. In Tables \ref{tab:t1}-\ref{tab:t4}, $h=\Delta x$ is the stepsize and $N=1/h$ is the number of elements used in each experiment. 

\begin{table}[ht!]
\centering
{\small
\begin{tabular}{c|cccc|cccc}
\hline
$N$ & $E_0(\eta;h)$ & $\tilde{r}_0$ &  $ E_0(u;h)$ & $\tilde{r}_0$ & $E_1(\eta;h)$ & $\tilde{r}_1$ &  $ E_1(u;h)$ & $\tilde{r}_1$ \\
\hline
$ 10  $ & $ 7.4985\times 10^{-3} $ &  --  & $ 5.7527\times 10^{-2} $ &  --   & $ 8.6349\times 10^{-2} $ &  --  &       $ 1.9031\times 10^{-1} $ &  --   \\  
$ 20  $ & $ 1.8713\times 10^{-3} $ & $ 2.0025 $ & $ 1.4451\times 10^{-2} $ &  $ 1.9930 $  & $ 4.3198\times 10^{-2} $ & $ 0.9992 $ & $ 9.4663\times 10^{-2} $ &  $ 1.0075 $  \\
$ 40  $ & $ 4.6769\times 10^{-4} $ & $ 2.0005 $ & $ 3.6176\times 10^{-3} $ &  $ 1.9981 $  & $ 2.1602\times 10^{-2} $ & $ 0.9997 $ & $ 4.7283\times 10^{-2} $ &  $ 1.0015 $  \\
$ 80  $ & $ 1.1691\times 10^{-4} $ & $ 2.0001 $ & $ 9.0471\times 10^{-4} $ &  $ 1.9995 $  & $ 1.0802\times 10^{-2} $ & $ 0.9999 $ & $ 2.3635\times 10^{-2} $ &  $ 1.0003 $  \\
$ 160 $ & $ 2.9228\times 10^{-5} $ & $ 2.0000 $ & $ 2.2619\times 10^{-4} $ &  $ 1.9999 $  & $ 5.4010\times 10^{-3} $ & $ 1.0000 $ & $ 1.1817\times 10^{-2} $ &  $ 1.0001 $  \\
$ 320 $ & $ 7.3071\times 10^{-6} $ & $ 2.0000 $ & $ 5.6550\times 10^{-5} $ &  $ 2.0000 $  & $ 2.7005\times 10^{-3} $ & $ 1.0000 $ & $ 5.9085\times 10^{-3} $ &  $ 1.0000 $  \\
$ 640 $ & $ 1.8267\times 10^{-6} $ & $ 2.0000 $ & $ 1.4137\times 10^{-5} $ &  $ 2.0000 $ & $ 1.3502\times 10^{-3} $ & $ 1.0000 $ & $ 2.9542\times 10^{-3} $ &  $ 1.0000 $ \\
\hline
\end{tabular}
}
\caption{\!$L^2$ and $H^1$ errors and the corresponding experimental convergence rates for Lagrange linear elements $S_h(0,2)$.}\label{tab:t1}
%
\vspace*{3mm}
{\small
\begin{tabular}{c|cccc|cccc}
\hline
$N$ & $E_0(\eta;h)$ & $\tilde{r}_0$ &  $ E_0(u;h)$ & $\tilde{r}_0$ & $E_1(\eta;h)$ & $\tilde{r}_1$ &  $ E_1(u;h)$ & $\tilde{r}_1$ \\
\hline
$ 10  $ & $ 1.7877\times 10^{-4} $ &  --  &       $ 1.3838\times 10^{-3} $ &  --   & $ 3.5035\times 10^{-3} $ &  --  &       $ 2.1586\times 10^{-2} $ &  --   \\  
$ 20  $ & $ 2.2304\times 10^{-5} $ & $ 3.0027 $ & $ 1.7764\times 10^{-4} $ &  $ 2.9616 $  & $ 8.7624\times 10^{-4} $ & $ 1.9994 $ & $ 5.5378\times 10^{-3} $ &  $ 1.9627 $  \\  
$ 40  $ & $ 2.7865\times 10^{-6} $ & $ 3.0008 $ & $ 2.2347\times 10^{-5} $ &  $ 2.9908 $  & $ 2.1906\times 10^{-4} $ & $ 2.0000 $ & $ 1.3930\times 10^{-3} $ &  $ 1.9911 $  \\  
$ 80  $ & $ 3.4826\times 10^{-7} $ & $ 3.0002 $ & $ 2.7978\times 10^{-6} $ &  $ 2.9977 $  & $ 5.4764\times 10^{-5} $ & $ 2.0000 $ & $ 3.4879\times 10^{-4} $ &  $ 1.9978 $  \\  
$ 160 $ & $ 4.3531\times 10^{-8} $ & $ 3.0001 $ & $ 3.4987\times 10^{-7} $ &  $ 2.9994 $  & $ 1.3691\times 10^{-5} $ & $ 2.0000 $ & $ 8.7232\times 10^{-5} $ &  $ 1.9994 $  \\  
$ 320 $ & $ 5.4413\times 10^{-9} $ & $ 3.0000 $ & $ 4.3738\times 10^{-8} $ &  $ 2.9999 $  & $ 3.4227\times 10^{-6} $ & $ 2.0000 $ & $ 2.1810\times 10^{-5} $ &  $ 1.9999 $  \\  
$ 640 $ & $ 6.8018\times 10^{-10}$ & $ 3.0000 $ & $ 5.4675\times 10^{-9} $ &  $ 2.9999 $ & $ 8.5569\times 10^{-7} $ & $ 2.0000 $ & $ 5.4526\times 10^{-6} $ &  $ 2.0000 $ \\
\hline
\end{tabular}
}
\caption{$L^2$ and $H^1$ errors and the corresponding experimental convergence rates for Lagrange quadratic elements $S_h(0,3)$.}\label{tab:t2}
%
%
\vspace*{3mm}
{\small
\begin{tabular}{c|cccc|cccc}
\hline
$N$ & $E_0(\eta;h)$ & $\tilde{r}_0$ &  $ E_0(u;h)$ & $\tilde{r}_0$ & $E_1(\eta;h)$ & $\tilde{r}_1$ &  $ E_1(u;h)$ & $\tilde{r}_1$ \\
\hline
$ 10  $ & $ 3.0305\times 10^{-6} $ &  --  &       $ 8.8820\times 10^{-5} $ &  --   & $ 9.3259\times 10^{-5} $ &  --  &       $ 2.0256\times 10^{-3} $ &  --   \\  
$ 20  $ & $ 1.8939\times 10^{-7} $ & $ 4.0001 $ & $ 5.5572\times 10^{-6} $ &  $ 3.9985 $ & $ 1.1655\times 10^{-5} $ & $ 3.0002 $ & $ 2.5351\times 10^{-4} $ &  $ 2.9982 $  \\  
$ 40  $ & $ 1.1838\times 10^{-8} $ & $ 3.9999 $ & $ 3.4742\times 10^{-7} $ &  $ 3.9996 $ & $ 1.4570\times 10^{-6} $ & $ 2.9999 $ & $ 3.1699\times 10^{-5} $ &  $ 2.9995 $  \\  
$ 80  $ & $ 7.3993\times 10^{-10} $ &$ 3.9999 $ & $ 2.1715\times 10^{-8} $ &  $ 3.9999 $ & $ 1.8214\times 10^{-7} $ & $ 3.0000 $ & $ 3.9628\times 10^{-6} $ &  $ 2.9999 $  \\  
$ 160 $ & $ 4.6402\times 10^{-11} $ &$ 3.9951 $ & $ 1.3575\times 10^{-9} $ &  $ 3.9996 $  & $ 2.2767\times 10^{-8} $ & $ 3.0000 $ & $ 4.9536\times 10^{-7} $ &  $ 3.0000 $  \\  
\hline
\end{tabular}
}
\caption{$L^2$ and $H^1$ errors and the corresponding experimental convergence rates Lagrange cubic elements $S_h(0,4)$.}\label{tab:t3}
\end{table}
\begin{table}[ht!]
\vspace*{3mm}
{\small 
\begin{tabular}{c|cccc|cccc}
\hline
$N$ & $E_0(\eta;h)$ & $\tilde{r}_0$ &  $ E_0(u;h)$ & $\tilde{r}_0$ & $E_1(\eta;h)$ & $\tilde{r}_1$ &  $ E_1(u;h)$ & $\tilde{r}_1$ \\
\hline
$ 10  $ & $ 3.2731\times 10^{-4} $ &  --  &       $ 6.7353\times 10^{-5} $ &  --          & $ 2.0410\times 10^{-2} $ &  --  &       $ 4.0162\times 10^{-3} $ &  --   \\  
$ 20  $ & $ 2.0717\times 10^{-5} $ & $ 3.9818 $ & $ 4.0005\times 10^{-6} $ &  $ 4.0735 $  & $ 2.6131\times 10^{-3} $ & $ 2.9654 $ & $ 4.9751\times 10^{-4} $ &  $ 3.0130 $  \\  
$ 40  $ & $ 1.3119\times 10^{-6} $ & $ 3.9810 $ & $ 2.4904\times 10^{-7} $ &  $ 4.0057 $  & $ 3.3180\times 10^{-4} $ & $ 2.9774 $ & $ 6.2746\times 10^{-5} $ &  $ 2.9872 $  \\  
$ 80  $ & $ 8.2686\times 10^{-8} $ & $ 3.9880 $ & $ 1.5640\times 10^{-8} $ &  $ 3.9930 $  & $ 4.1841\times 10^{-5} $ & $ 2.9873 $ & $ 7.9071\times 10^{-6} $ &  $ 2.9883 $  \\  
$ 160 $ & $ 5.1918\times 10^{-9} $ & $ 3.9933 $ & $ 9.8173\times 10^{-10} $ &  $ 3.9938 $  & $ 5.2545\times 10^{-6} $ & $ 2.9933 $ & $ 9.9334\times 10^{-7} $ &  $ 2.9928 $  \\  
$ 320 $ & $ 3.2535\times 10^{-10} $ & $ 3.9962 $ & $ 6.1521\times 10^{-11} $ & $ 3.9962 $  & $ 6.5837\times 10^{-7} $ & $ 2.9966 $ & $ 1.2450\times 10^{-7} $ &  $ 2.9961 $  \\  
\hline
\end{tabular}
}
\caption{$L^2$ and $H^1$ errors and the corresponding experimental convergence rates for  cubic spline elements $S_h(2,4)$.}\label{tab:t4}
\end{table}

All of our  results confirm the theoretical, optimal error estimates that we proved in Section \ref{sec:estimates}. Note that in the case of cubic Lagrange elements (see Table \ref{tab:t3}) the method converges faster compared to the case of cubic splines, and the errors (especially those related to the $L^2$-norm) reach the limits of the numerical method quickly, resulting to very small errors of order $10^{-11}$ for $N\geq 160$.

\subsection{The effect of the boundary conditions}

We now explore the influence of the boundary conditions to the reflection of solitary waves on a vertical and impenetrable wall in comparison to other Boussinesq-type systems. In particular, we perform a series of numerical experiments for four different Boussinesq-type systems, namely, the regularized Nwogu system with $\theta^2=1/4$ and coefficients given by~\eqref{eq:coeffs2}, the classical Nwogu system, and the BBM-BBM \cite{BCS2002} and Serre \cite{serre} systems. For the BBM-BBM system, we use the conservative fully-discrete numerical method of \cite{MRKS2021}. For the Serre equations, we use the results of \cite{MSM2017}. For the Nwogu system, we test two numerical methods: (i) The numerical method of~\cite{WB1999} based on a modified Galerkin semidiscretization, and (ii) a new standard Galerkin method accompanied with the fourth-order, four-stage, classical Runge-Kutta method with stepsize $\Delta t=\Delta x/10$. Both numerical methods (i) and (ii) appear to converge with optimal convergence rate $4$ to the velocity $u$. For the free surface elevation, we obtained clear convergence rates only in the case of the standard Galerkin method (ii) and the rates were suboptimal and equal to $2$. The modified Galerkin method of (i) appears to have better errors in $L^2$ and $H^1$ compared to (ii). Both methods deserve further exploration, and this is reserved as a topic for future work. 

\begin{figure}[ht!] 
\centering 
\includegraphics[width=0.8\columnwidth]{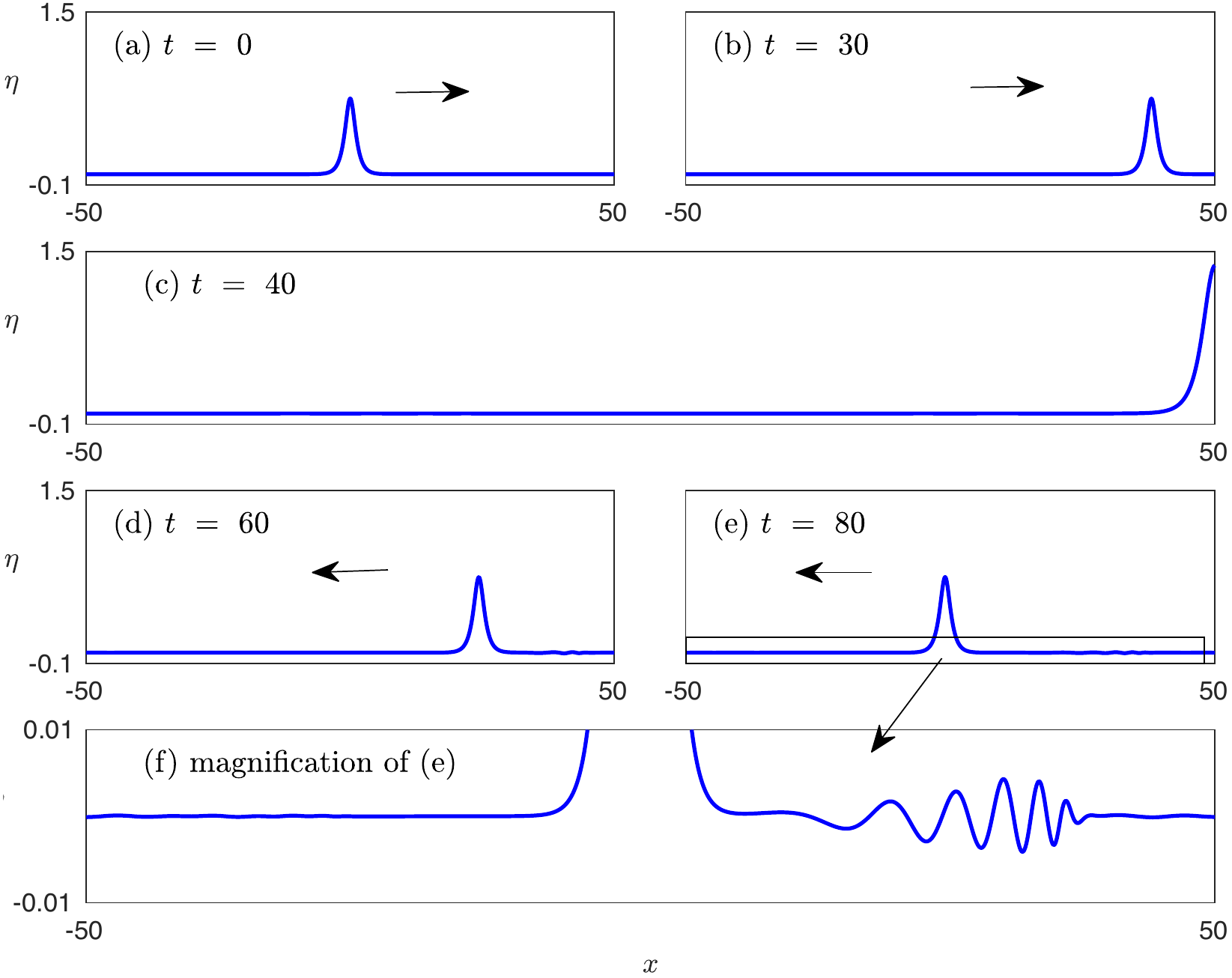}
\caption{Reflection of a solitary wave with amplitude $A\approx 0.7$ by a vertical wall for the Nwogu system.}
\label{fig:figure0}
%
\vspace*{5mm}
\includegraphics[width=0.7\columnwidth]{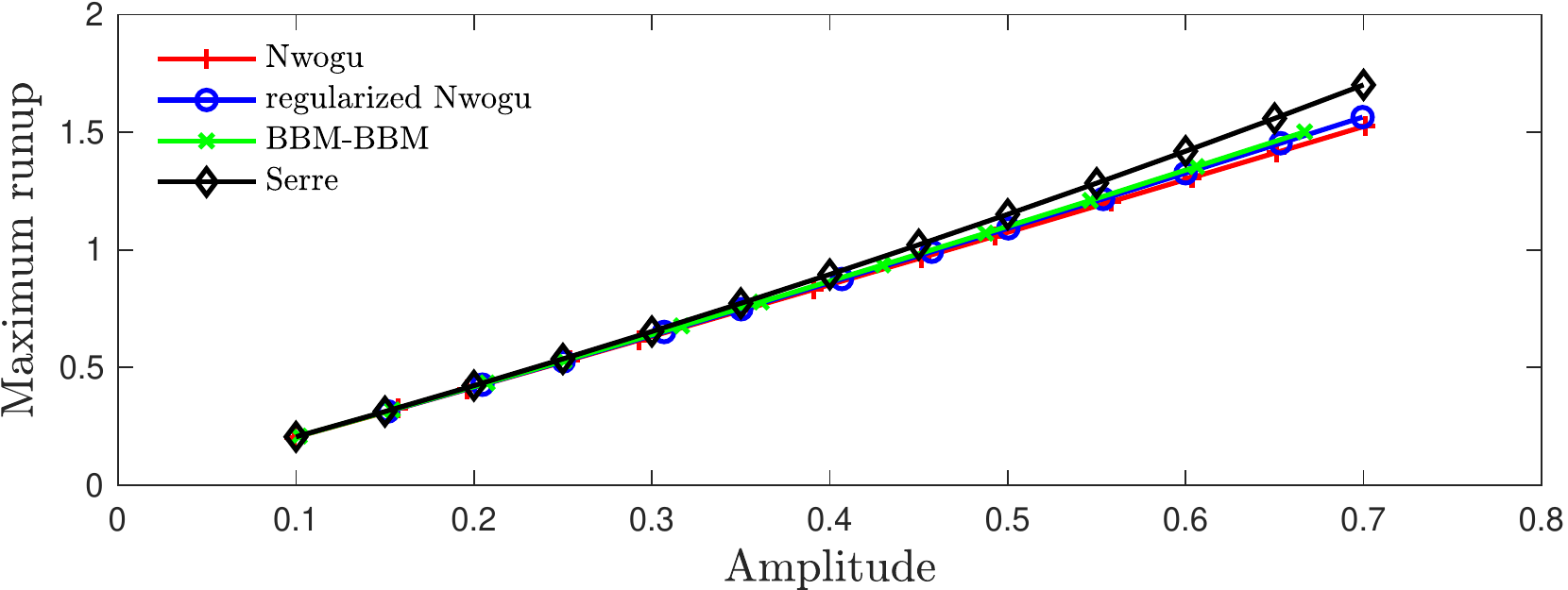}
\caption{Maximum runup of solitary waves on a vertical wall. Comparison between the Nwogu, BBM-BBM and Serre systems of equations.}
\label{fig:figure1}
\end{figure}

In order to study the influence of the boundary conditions on the reflection of solitary waves, we consider solitary waves of amplitudes approximately $0.1,0.15,0.2,\ldots,0.7$ initially placed symmetrically around $x=0$ in the interval $[-50,50]$. While these solitary waves propagate towards the  wall located at $x=50$, we allow quite large time intervals of numerical integration so as to study their complete reflection by the wall. One such reflection is depicted in Figure \ref{fig:figure0}. As there are no exact formulas for solitary waves of the Boussinesq systems of either Nwogu or BBM-BBM type, we use numerically generated solitary waves. To generate solitary waves for the BBM-BBM system, we use the Petviashvilli method as it is described in \cite{MRKS2021}. For Nwogu's system, we use the Levenberg-Marquardt modification of Newton's algorithm to minimize the residuals to the equations resulted in pseudo-spectral discretizations of the respective ordinary differential equation. The numerical initial conditions are obtained via interpolation as $(\tilde{\eta}(0),\tilde{u}(0))\in S_h\times S_h^0$. This technique results in the computation of the initial profiles of the solitary waves with spectral accuracy.

Note that, as in other Boussinesq systems, the reflection of the solitary waves of the Nwogu system is not elastic \cite{AD2012}. In Figure \ref{fig:figure0}, we observe that, as a result of the inelastic collision, dispersive tails are generated following the solitary pulse. Also, the solitary pulse returns modified after the reflection since part of its energy has been transferred to the dispersive tails.

Figure \ref{fig:figure1} presents the maximum wave-height (runup on the wall) recorded during the reflection of the solitary waves. Observe that all three Boussinesq models describe the reflection in a similar manner, while all systems converge to the the same solutions as the amplitude becomes small. Apparently, the Nwogu system  is the least accurate compared to the other two systems. The Serre system is known to be the most accurate since it does not describe only small amplitude waves. The BBM-BBM system, which is well-posed only with the conditions $\eta_x(\pm L,t)=u(\pm L,t)=0$ on the boundary (the condition $u_{xx}(\pm L,t)=\text{constant}$ is implied), appears to have slightly better performance compared to the other two systems.

\section{Conclusions}
\label{conc-s}

A theoretical and numerical study of an initial-boundary value problem for regularized Nwogu-type Boussinesq systems with reflective boundary conditions was presented. The well-posedness of this problem requires, in addition to the standard wall-boundary conditions, a condition on the second derivative of the velocity on the boundary wall. These boundary conditions, which are also satisfied by the solutions of the Euler equations, were tested numerically against other Boussinesq systems. Furthermore, the above choice of boundary conditions was justified by deriving a novel solution formula for the linearized problem using the unified transform of Fokas. The numerical method, which is based on finite element methods, was shown to be convergent with optimal rates.



\appendix

\section{Justification of wall-boundary conditions in Boussinesq-type equations}\label{sec:appendix}
In this section, we show that the set of boundary conditions we use in this work is also satisfied by the solutions of the Euler equations in the same bounded domain with wall-boundary conditions. Some of the results of this section have been presented in \cite{Khakimzyanov2018a}; here, we repeat them for the sake of completeness and, furthermore, we  extend them  in order to cover the boundary conditions of the previous sections. 

Let $\Omega=\left\{(x,y): -L<x<L, -D\leq y\leq \eta(x,t)\right\}$ where $D$ denotes the depth of the ocean floor measured from the rest position (zero level) of the water's free surface. For $(x,y)\in\Omega$ and $t\geq 0$, the Euler equations of irrotational perfect fluid flow can be written in the form
\begin{align}
&u_t+uu_x+vu_y+\frac{1}{\rho} \, p_x=0\ ,
\quad
v_t+uv_x+vv_y+\frac{1}{\rho} \, p_y=-g\ ,\label{eq:mom}
\\
&u_x+v_y = 0\ ,\label{eq:mass}\\
&u_y-v_x=0\ ,\label{eq:irot}
\end{align}
where $u=u(x,y,t)$ and $v=v(x,y,t)$ are the horizontal and vertical components of the fluid velocity measured at $(x,y)\in \Omega$, $\rho$ denotes the density of the fluid, and $g$ the gravitational acceleration. Equations~\eqref{eq:mom} and \eqref{eq:mass} correspond to  momentum and  mass  conservation, respectively, while equation~\eqref{eq:irot} is the irrotationality condition. The impermeability of the ocean floor can be expressed by the condition $v(x,y,t)=0$ at $y=-D$, while the impenetrability of the wall is expressed by the boundary condition
\begin{equation}\label{eq:wallbcs}
u(-L,y,t)=u(L,y,t)=0\ .
\end{equation}
The free surface boundary conditions are written as
\begin{equation}\label{eq:pressbc}
\eta_t+u\eta_x=v \quad \text{and} \quad p(x,y,t)=p_\text{atm}\quad \text{at}\quad y=\eta(x,t)\ , 
\end{equation}
where $p_\text{atm}$ denotes the atmospheric pressure.
Evaluating the first of equations \eqref{eq:mom} at $x=\pm L$ and using the wall-boundary conditions \eqref{eq:wallbcs}, we find
\begin{equation}\label{eq:npress}
    p_x(\pm L,y,t)=0\ .
\end{equation}
Differentiating the second of conditions \eqref{eq:pressbc} with respect to $x$ and using \eqref{eq:npress}, we obtain
$\eta_x(\pm L,t) \, p_y(\pm L,\eta(\pm L,t),t)=0$,
which implies the additional boundary condition
\begin{equation}
\eta_x(L,t)=\eta_x(-L,t)=0
\end{equation}
since $p_y(\pm L,\eta(\pm L,t),t)$ is not necessarily zero.
Similarly, differentiating the irrotationality condition~\eqref{eq:irot} with respect to $y$, we obtain
\begin{equation}\label{eq:newbc1}
    v_{yx}(\pm L,y,t)=u_{yy}(\pm L,y,t)=0\ .
\end{equation}
Thus, differentiating the mass conservation \eqref{eq:mass} with respect to $x$ and then using \eqref{eq:newbc1}, we find
\begin{equation}
u_{xx}(\pm L,y,t)=-v_{yx}(\pm L,y,t)=0\ .
\end{equation}
This justifies the wall-boundary conditions used for Boussinesq-type systems here and in \cite{ADM2009}.

Similar results are true in three-dimensional domains as well,  independently of the bottom topography. For example, if $\bu(\bx,z,t)$ denotes the velocity of the fluid at position $(\bx,z)=(x,y,z)$  and time $t$ which satisfies   the three-dimensional Euler equations with slip-wall boundary conditions $\bu(\bx,z,t)\cdot \bn=0$ on the boundary  $\bx\in\partial \Omega$, $z\in[-D,\eta]$, then  the Neumann boundary conditions are $\nabla\eta(\bx,t)\cdot\bn=0$ and $\nabla(\Div\bu(\bx,z,t))\cdot\bn=0$ on $\partial\Omega$. It is noted that the function that describes the depth (bottom topography) does not appear in any of these boundary conditions.

\vspace*{3mm}
\noindent
\textbf{Acknowledgements.} This work was partially supported by a grant from the National Science Foundation (NSF-DMS 2206270 to D. Mantzavinos).



\bibliographystyle{plain}
\bibliography{references}
%
%
%
%
\end{document}